\documentclass[11pt,reqno]{amsart}
\usepackage{amssymb}
\usepackage{amsmath}
\usepackage{amsfonts}
\usepackage{graphicx}
\usepackage{subcaption}
\usepackage{tikz}
\usetikzlibrary{calc}
\usepackage{pgfplots}
\usepgfplotslibrary{groupplots}
\pgfplotsset{compat=1.18}
\usepackage{mathtools}
\usepackage[width=150mm,height=220mm,
centering,
]{geometry}
\usepackage{cite}
\usepackage{algorithm,algorithmic}
\usepackage{booktabs}
\usepackage{array}
\allowdisplaybreaks

\begin{document}
	\newtheorem{theorem}{Theorem}[section]
	\newtheorem{lemma}[theorem]{Lemma}
	\newtheorem{definition}[theorem]{Definition}
	\newtheorem{example}[theorem]{Example}
	\newtheorem{corollary}[theorem]{Corollary}
	\newtheorem{remark}[theorem]{Remark}
	\newtheorem{proposition}[theorem]{Proposition}
	\numberwithin{equation}{section}
	\def\arg{\mathrm{arg}}
	\def\rmd{\mathrm d}

	\bibliographystyle{abbrv}
	
	\title
	{Algorithmic Implementation of Multi-Window STFT Phase Retrieval }

\author[T. Chen]{Ting Chen}
\address{School of Mathematical Sciences and LPMC,
Nankai University,
Tianjin,
China}
\email{t.chen@nankai.edu.cn}

\author[H. Lu]{Hanwen Lu}
\address{School of Mathematical Sciences and LPMC,
Nankai University,
Tianjin,
China}
\email{2013537@mail.nankai.edu.cn}

\author[W. Sun]{Wenchang Sun}
\address{School of Mathematical Sciences and LPMC,
Nankai University,
Tianjin,
China}
\email{sunwch@nankai.edu.cn}

\author[Y. Zhao]{Yutong Zhao}
\address{School of Mathematical Sciences and LPMC,
Nankai University,
Tianjin,
China}
\email{zhaoyt@mail.nankai.edu.cn}

\date{}

	
\thanks{Corresponding author: Wenchang Sun.}

\thanks{This work was supported by the National Natural Science Foundation of China (grant No. 12271267 and 12571104) and the Fundamental Research Funds for the Central Universities.}

	\begin{abstract}
STFT phase retrieval aims to reconstruct a signal from spectrogram magnitudes. While two compactly supported windows theoretically guarantee uniqueness, a complete algorithmic framework remains lacking. This paper systematically constructs numerical schemes for multi-window STFT phase retrieval. We theoretically provide a nonuniqueness counterexample for the single-window case. We then propose a two-window reconstruction algorithm based on Shannon interpolation and analytic square-root extraction. To overcome the instability of the global sign choice, we introduce a third window and a robust $K$-sign method that determines the sign locally. We also adapt SCR algorithm to the STFT setting. For the $K$-sign method, we establish deterministic stability on a finite contiguous block of frequency samples, with the contribution of the complementary tail explicitly included in the error bound. Numerical experiments validate the effectiveness of all schemes.
	\end{abstract}
	
	\maketitle
	
	\section{Introduction}
	The reconstruction of a signal from intensity-only measurements, known as phase retrieval, is a central inverse problem in science and engineering. It arises naturally in numerous physical contexts, including optics \cite{1963optics,1978object,1988optics,2015Optical}, X-ray crystallography \cite{1983cry,2014cry,2018cry}, astronomy \cite{1987ast,2014ast}, and quantum information theory \cite{2013Quantum,2015quantum}. When the measurement system is modeled by the short-time Fourier transform (STFT) \cite{2021stft,2022gabor,
2018stft,		2022stft,2025stft,2025stft1,
2019gabor,2021stft1,2024stlct,
2004entire,Jaganathan2016STFTPR,
Li2017PhaseRetrieval,
Alaifari2024MultiWindow,
Chen2026STFTACHA}, the task is to reconstruct a signal from its spectrogram  $|\mathcal V_\phi f(t,\omega)| $.
	
	Here $\mathcal V_\phi f$ denotes the short-time Fourier transform. Let $f\in L^2(\mathbb R)$ be a signal and $\phi\in L^2(\mathbb R)$ be a window function.
	The short-time Fourier transform (STFT) of $f$ with respect to $\phi$ is defined as
	\[
	\mathcal{V}_\phi f(t,\omega):=\int_{\mathbb R} f(x)\,\overline{\phi(x-t)}\,e^{-2\pi i\omega x}\,dx,\qquad (t,\omega)\in\mathbb R^2.
	\]
	The STFT provides a time-frequency representation of $f$. As is well known, under mild conditions on the window $\phi$, the STFT magnitude $|\mathcal{V}_\phi f(\Lambda)|$ over $\Lambda=\mathbb R^2$ determines every $f\in L^2(\mathbb R)$ uniquely up to a global phase~\cite{2021stft,2019gabor}.
	
	However, in practical applications, measurements are typically available only on a discrete set $\Lambda\subset\mathbb R^2$, leading to the discrete STFT phase retrieval problem: recover $f$ from the data
	\[
	\{|\mathcal{V}_\phi f(t,\omega)|:(t,\omega)\in\Lambda\}.
	\]
	It is well known that $f$ and $\lambda f$ are indistinguishable from the data, where $\lambda\in\mathbb T:=\{z\in\mathbb C:|z|=1\}$. Hence we only recover $f$ up to a global phase.
	
	The uniqueness of STFT phase retrieval has been studied extensively. For finite-dimensional spaces, Jaganathan, Eldar and Hassibi \cite{Jaganathan2016STFTPR} showed that almost every signal is uniquely determined up to a global phase with overlapping windows, yet single-window schemes always admit nontrivial ambiguities. To overcome this, Li, Cheng, Han et al. \cite{Li2017PhaseRetrieval} introduced a multi-window graph-theoretic framework, proving that connectivity of a certain support graph guarantees uniqueness, and provided an algebraic reconstruction algorithm with stability.
	
	For infinite-dimensional function spaces, recent work by Grohs and Liehr \cite{2022stft} further revealed a fundamental barrier: for any window and any lattice, there exist distinct functions sharing the same STFT samples. Then Grohs, Liehr and Rathmair \cite{2025stft1} showed that four windows suffice to uniquely determine every
	$L^2$-function from lattice samples. Alaifari and Yang \cite{Alaifari2024MultiWindow} constructed multi-window families whose ambiguity functions cover the time-frequency plane, enabling stable recovery via direct inversion.
We \cite{Chen2026STFTACHA} proved that two compactly supported windows ensure unique recovery for $L^p$-functions under discrete sampling, yet a complete algorithm was missing.
	
We observe that the key
to STFT phase retrieval with compactly supported windows
lies in the conjugate phase retrieval problem in the Paley-Wiener space. Within $\mathrm{PW}_B^p$ itself, however, a closely related obstruction appears. McDonald's earlier work \cite{2004entire} provided uniqueness and factorization results for entire functions, but not an explicit reconstruction algorithm. Subsequently, Lai, Littmann, and Weber \cite{2021Conjugate} studied the conjugate phase retrieval problem from the measurements
	\[
	\bigl\{|f(t_n)|,\ |f(t_n+b)-f(t_n)|:n\in\mathbb Z\bigr\}.
	\]
	They showed that the squared modulus can be stably reconstructed from these samples and that the measurement map is injective. However, they explicitly remarked that no obvious reconstruction algorithm was available from these two measurement chains alone. Their subsequent numerical scheme therefore resorted to additional structured convolution measurements, namely the structured convolution reconstruction (SCR) algorithm.
	
More recently, Cheng, Wu, and Xian \cite{Cheng2025Conjugate}
developed a graph-based framework for conjugate phase retrieval,
showing that any complex-valued function in the space $\mathrm{PW}_B^2$ can be recovered,
up to a global phase and conjugation,
from structured phaseless samples taken at three times the Nyquist rate, under appropriate conditions on the sampling set.
They also
introduced two numerical reconstruction algorithms for this purpose.

In this paper, we first
present a counterexample, which can be regarded as a supplement to the counterexample in Grohs and Liehr \cite{2022stft},
to show that for a certain window function
and the union of finitely many translates of any lattice in $\mathbb R^2$, including the case where the translation parameters are irrational multiples of the lattice spacing,
there exists an $L$-nonseparable  signal
for which the absolute values of the STFT samples at these points do not  uniquely determine the signal up to a phase.
Leveraging the uniqueness conditions established in \cite{Chen2026STFTACHA}, we design a two-window reconstruction algorithm based on the analyticity of Paley-Wiener functions and the Shannon sampling theorem. However, this algorithm is highly sensitive to noise, particularly in the analytic square-root step, which leads to instability. To mitigate this, we propose a more robust three-window reconstruction strategy.
Moreover, we give explicit error bounds for the stability of the reconstruction algorithm.
Furthermore, by incorporating the conjugate phase retrieval framework of Lai, Littmann, and Weber  \cite{2021Conjugate}, we construct another three-window scheme via the SCR algorithm.
	
	\section{Preliminaries and Single-Window Counterexample}
	This section establishes the necessary background and demonstrates a fundamental limitation of single-window STFT phase retrieval.
	
	\subsection{Basic results}
	In this subsection, we introduce
	some fundamental tools from sampling theory and phase retrieval that are used in our algorithms.

	For $1\le p\le\infty$ and bandwidth $B>0$, define the Paley-Wiener space
	\[
	\mathrm{PW}_B^p:=\left\{f\in L^p(\mathbb R): \operatorname{supp}\hat f\subset[-B,B]\right\},
	\]
	where $\hat f$ is the Fourier transform of $f$ in the distributional sense.
	By \cite[Theorem 4]{Entire2014}, $\mathrm{PW}_B^p$ consists of all entire functions of exponential type $2\pi B$ whose restriction on the real line is in $L^p(\mathbb R)$.
	
	For $1\le p<\infty$, the classical Shannon sampling theorem says that every $f\in \mathrm{PW}_B^p$ can be recovered from its samples at the Nyquist rate; see Proposition~\ref{prop2}. The endpoint case $p=\infty$ is treated separately below; see Proposition~\ref{prop2.3}. These formulas are applied in our reconstruction algorithms.
	\begin{proposition}[{\cite[Theorem 4.1]{pesenson2025notes}}]\label{prop2}
		Let $1\le p<\infty$ and $B>0$. Then for any $f \in \mathrm{PW}_B^p$, we have
		\begin{align}\label{equ1}
			f(x)=\sum_{n \in \mathbb{Z}} f\left(\frac{n}{2 B}\right) \operatorname{sinc}(2 B x-n), \quad \forall x \in \mathbb{R}
		\end{align}
		with absolutely and uniformly convergence on $\mathbb{R}$, where $\operatorname{sinc}(x)=\frac{\sin(\pi x)}{\pi x}$.
	\end{proposition}
	
	%\begin{proposition}[{\cite[Theorem 4.3]{pesenson2025notes}}]\label{prop2.2}
	%If $ f \in \mathrm{PW}_B^\infty $, then
	%for any $A>B$ and $ z \in \mathbb{C} $ one has
	%\begin{align}\label{equ2.1}
	%f(z) = \sum_{n \in \mathbb{Z}} f\!\left(\frac{n}{2A}\right) \operatorname{sinc}\!\left(2A z - n\right),
	%\end{align}
	%where the series uniformly converges on compact subsets of $ \mathbb{C} $.
	%\end{proposition}
	
	For $p=\infty$, the series in \eqref{equ1} may fail to converge pointwise.  For functions with integrable Fourier transforms, we instead use the Ces\`aro (Fej\'er) means.
	\begin{proposition}\label{prop2.3}
		Let $f\in PW_B^\infty$ and assume in addition that
		$\hat f\in L^1(\mathbb R)$. Then we have
		\begin{align}\label{equ2.2}
			f(x) = \lim_{N\to\infty}
			\sum_{|n|\le N}
			\left(1-\frac{|n|}{N+1}\right)
			f\left(\frac{n}{2 B}\right)
			\operatorname{sinc} (2Bx-n ),
		\end{align}
	where the convergence is uniform on $\mathbb R$.	
	\end{proposition}
	Furthermore, the following result, which is a variant of \cite[Lemma 2.3]{Chen2026STFTACHA}, establishes a conjugate phase retrieval theorem for $\mathrm{PW}_B^p$ and is of significant importance to our algorithm design.
	\begin{proposition}\label{prop3}
		Let $1\le p\le\infty$, $B>0$ and
		$\{\omega_n\}_{n\ge 1}$ be either an enumeration of $\{k/(4B)\}_{k\in\mathbb Z}$
		or a sequence of increasing positive numbers such that
		\begin{equation}\label{eq:ea1}
			\limsup_{n\to\infty} \frac{n}{\omega_n} > 4B.
		\end{equation}
		Suppose $f,g\in \mathrm{PW}_B^p$ and $0<b\le 1/(2B)$. If
		\begin{equation}\label{eq:ea2}
			|f(\omega_n)| = |g(\omega_n)|,
			\quad
			|f(\omega_n+b)-f(\omega_n)|
			= |g(\omega_n+b)-g(\omega_n)|,  \quad n\ge 1,
		\end{equation}
		then there exists some $\lambda\in \mathbb{T}$ such that $f=\lambda g$ or $f=\lambda\overline{g}$.
	\end{proposition}
	
	Based on the above proposition, we \cite{Chen2026STFTACHA} proposed the uniqueness theorem for STFT phase retrieval with two compactly supported windows. Before stating the uniqueness theorem for STFT phase retrieval, we need the following notion.
	
	\begin{definition}
		A function $f$ is called $L$-separable if there exists an interval of length $L$ such that $f$ vanishes on this interval. If no such interval exists, we call $f$ $L$-nonseparable.
	\end{definition}
	
	With this definition in hand, we now state the two-window STFT uniqueness result, which forms the theoretical foundation for our algorithms.
	
	\begin{proposition}[{\cite[Theorem 1.1]{Chen2026STFTACHA}}]\label{prop1}
		Let $1\le p<\infty$ and $p'$ be its conjugate exponent. Suppose $B>0$, $\phi \in  L^{p'}(\mathbb{R})$
		with $\operatorname{supp}( \phi) \subset
		[-B, B]$, $\overline{\phi(-x)}=\phi(x)$ and
		\begin{align*}
			\phi(x) \ne 0,\quad \mbox{ a.e. } x\in [-B, B].
		\end{align*}
		Define $\psi(x)=\phi(x) (e^{ 2\pi ixb}-1)$ with $0<b\le 1/(2B)$.
		Suppose that $\{\omega_n\}_{n\ge 1}$
		be either  an enumeration of $\{k/(4B)\}_{k\in\mathbb Z}$
		or a sequence of increasing
		positive numbers
		satisfying (\ref{eq:ea1}). Then for any $0<a\le B$ and $f\in L^p(\mathbb R)$ which is $(2B-a)$-nonseparable, $f$ is determined up to a global phase by the measurements
		\begin{align} \label{eq:ed2}
			\{|\mathcal{V}_{\phi} f(ma,\omega_{n})|,~|\mathcal{V}_{\psi} f(ma,\omega_{n})|:n,m \in \mathbb{Z}\}.
		\end{align}
	\end{proposition}
	For convenience, let [$f$] denote
	$\{g\in\mathrm{PW}_B^p:g=\lambda f~or~\lambda\overline{f}\}$ or $\{g\in L^p(\mathbb R):g=\lambda f\}$.
	\subsection{A Theoretical Counterexample for Single-Window Non-Uniqueness}
	The uniqueness result in Proposition~\ref{prop1} relies on two compactly supported windows and the sampling lattice $\Lambda=(t, \omega)=(ma, \omega_n)$. For the rectangular window $\phi=\chi_{[-B,B]}$, Bartusel's characterization \cite[Theorem 5.1.8 and Example 5.4.3]{Bartusel2023} shows that full-plane STFT magnitudes uniquely determine any $L$-nonseparable signal in $L^2(\mathbb{R})$ up to a global phase whenever $0<L\le 2B $. In contrast, Grohs and Liehr \cite[Theorems 1.2-1.3]{2022stft} proved that nonuniqueness is unavoidable for any single $L^2$-window when sampling on a translated lattice or an equally spaced family of parallel lines. This naturally raises the question whether a single compactly supported window alone, combined with such more general discrete sampling sets that include irrational shifts, suffices for phase retrieval.
	
	In this section, we answer this question negatively. For the rectangular window and every $L>0$, nonuniqueness persists among $L$-nonseparable signals under joint measurements on the union of finitely many arbitrary translates of a common time grid, even when all frequencies are available. In particular, the shifts may differ by irrational multiples of the grid spacing, in which case their union cannot be contained in any single finer uniform time grid. The precise statement is as follows.
	\begin{theorem}
		Suppose that $B > 0$ is a constant and $\phi(x) = \chi_{[-B,B]}(x)$.
		Let $a$ be any positive number and let $T = \{t_1, \dots, t_m\}$ be a finite set with $0 < t_1 < \dots < t_m < a$.
		Define
		\[
		\Lambda = \bigcup_{t \in T} (t + a\mathbb{Z}) \times \mathbb{R}.
		\]
		Then for any $L > 0$, there exist $f, g \in L^p(\mathbb{R})$ which are $L$-nonseparable, satisfy
		\[
		|\mathcal{V}_\phi f(t,\omega)|=|\mathcal{V}_\phi g(t,\omega)|, \quad \forall
		(t,\omega)\in \Lambda,
		\]
		but they are not equal up to a global phase.
	\end{theorem}
	
	\begin{proof}
		For each $t_i \in T$, define the endpoint sets
		\[
		\mathcal{T}_i = \{t_i \pm B + an : n \in \mathbb{Z}\}.
		\]
		Define the union of all these endpoints as
		\[
		\mathcal{T} = \bigcup_{i=1}^{m} \mathcal{T}_i.
		\]
		The set $\mathcal{T}$ is a discrete subset of $\mathbb{R}$ with no accumulation points.
		And it partitions the real line into a collection of disjoint open intervals together with the points of $\mathcal{T}$ themselves.
		Enumerate these points in increasing order as $\{s_n\}_{n \in \mathbb{Z}}$ and denote $S_n = s_{n+1} - s_n$. The minimum distance between any two distinct points of $\mathcal{T}$ is
		\[
		D := \min_{\eta \ne \gamma \in \mathcal{T}} |\eta - \gamma| =\min_{n \in \mathbb{Z}} S_n > 0,
		\]
		and $D \le S_n \le a$ for all $n$.
		
		Given $L > 0$, set $\alpha = \min\{L/3,\, D/2\}$.
		Choose $r \in C_c^\infty(\mathbb{R})$ with $\operatorname{supp} r = [\alpha/2, \alpha]$ and $r > 0$ on $(\alpha/2, \alpha)$.
		Denote
		\[
		\left\{
		\begin{aligned}
			p(x)&=r(x)+2r(x+\alpha/2), \\
			q(x)&=2r(x)+r(x+\alpha/2).
		\end{aligned}
		\right.
		\]
		Then clearly $\operatorname{supp} p = \operatorname{supp} q = [0, \alpha]$ and
		their Fourier transforms satisfy
		\[
		\left\{
		\begin{aligned}
			\hat{p}(\omega) &= \hat{r}(\omega)(1 + 2e^{\pi i \alpha \omega}), \\
			\hat{q}(\omega) &= \hat{r}(\omega)(2 + e^{\pi i \alpha \omega}).
		\end{aligned}
		\right.
		\]
		Since $|1 + 2e^{i\theta}|^2 = |2 + e^{i\theta}|^2 = 5 + 4\cos\theta$, we have $|\hat{p}| = |\hat{q}|$ on $\mathbb{R}$, while $p$ and $q$ are not equal up to a global phase.
		
		Denote $c_n=2^{-|n|}$ and set $M_n = \lfloor S_n/(2\alpha) \rfloor \ge 1$, where$ \lfloor\cdot\rfloor$ is the floor function. Since $\alpha \le D/2$, we have $2\alpha \le D \le S_n$, and therefore $M_n \ge 1$ for all $n$. Define
		\[
		\left\{
		\begin{aligned}
			f(x) &= \sum_{n \in \mathbb{Z}} c_n \sum_{k=1}^{M_n} p(x - s_n - (2k-1)\alpha), \\
			g(x) &= \sum_{n \in \mathbb{Z}} c_n \sum_{k=1}^{M_n} q(x - s_n - (2k-1)\alpha).
		\end{aligned}
		\right.
		\]
		We claim that $f$ and $g$ satisfy the required properties.
		
		First, observe that the supports of the summands are pairwise disjoint. Indeed, for a fixed $n$, the $k$-th block is supported on $[s_n + (2k-1)\alpha, s_n + 2k\alpha]$, and consecutive blocks are separated by a gap of length $\alpha$. Moreover, the last block in $[s_n, s_{n+1}]$ ends at $s_n + 2M_n\alpha \le s_{n+1}$, while the first block in $[s_{n+1}, s_{n+2}]$ begins at $s_{n+1} + \alpha > s_{n+1}$. So blocks from adjacent intervals are also separated. Hence all blocks have pairwise disjoint supports, which implies $f, g \in L^p(\mathbb{R})$.
		
		Next, we verify the $L$-nonseparability. It suffices to check the gaps between adjacent intervals. From the support description above, the gap at each partition point $s_{n+1}$ has length
		\[
		\Delta_n = s_{n+1} + \alpha - (s_n + 2M_n\alpha) = S_n - 2M_n\alpha + \alpha.
		\]
		Since $M_n = \lfloor S_n/(2\alpha) \rfloor$, we have $$\Delta_n = S_n - 2M_n\alpha + \alpha < 2\alpha+\alpha=3\alpha.$$
		By $\alpha = \min\{L/3,\, D/2\}$, we have $3\alpha \le L$. And so $\Delta_n < L$. Hence $f$ and $g$ are $L$-nonseparable.
		
		Now we prove $|\mathcal{V}_\phi f| = |\mathcal{V}_\phi g|$ on $\Lambda$. For every fixed$ (t,\omega)\in\Lambda $, since the endpoints of the window $\phi$ are $t \pm B$, and $\mathcal{T}$ consists exactly of all points of the form $t_j \pm B + a\mathbb{Z}$, we have $t - B, t + B \in \mathcal{T}$. Denote $t-B = s_{n_{t_1}}$ and $t+B = s_{n_{t_2}}$. Then
		\begin{align*}
			\mathcal{V}_\phi f(t, \omega) &= \int_{\mathbb{R}} f(x) \chi_{[t-B,\,t+B]}(x) e^{-2\pi i x \omega}dx \\
			&= \hat{p}(\omega) \sum_{n=n_{t_1}}^{n_{t_2}-1} c_n \sum_{k=1}^{M_n} e^{-2\pi i (s_n + (2k-1)\alpha)\omega},
		\end{align*}
		and similarly,
		\[
		\mathcal{V}_\phi g(t, \omega) = \hat{q}(\omega) \sum_{n=n_{t_1}}^{n_{t_2}-1} c_n \sum_{k=1}^{M_n} e^{-2\pi i (s_n + (2k-1)\alpha)\omega}.
		\]
		Since $|\hat{p}(\omega)| = |\hat{q}(\omega)|$, we obtain $|\mathcal{V}_\phi f(t, \omega)| = |\mathcal{V}_\phi g(t, \omega)|$ for all $(t, \omega) \in \Lambda$.
		
		Finally, $f$ and $g$ are not equal up to a global phase, which follows immediately from the same property of $p$ and $q$ and the disjointness of the blocks. This completes the proof.
	\end{proof}
	
\section{Multi-Window Reconstruction Algorithms}
In this section, we develop explicit reconstruction algorithms for STFT phase retrieval with multiple compactly supported windows, building upon the uniqueness theory established in \cite{Chen2026STFTACHA}. More precisely, the uniqueness theorem in Proposition~\ref{prop1} asserts that, for a generic signal $f$, the magnitudes of the STFT with respect to two specially designed windows $\phi$ and $\psi$ determine $f$ up to a global phase. To turn this theoretical result into a constructive algorithm, we proceed in two stages. First, we develop reconstruction algorithms for the conjugate phase retrieval problem of Paley-Wiener space. Then, we incorporate these algorithms into the STFT framework.

\subsection{Reconstruction Algorithms For $\mathrm{PW}_B^p$.}
In this subsection, we study the reconstruction problem in the Paley-Wiener space $\mathrm{PW}_B^p$. The goal is to recover a function up to a global phase and conjugation from the magnitudes of its samples and finite differences. To this end, we develop two methods, namely the square-root method and the $K$-sign method.
\subsubsection{Square-Root Method}
We begin with the square-root method, which provides a basic reconstruction procedure in $\mathrm{PW}_B^p$. The detailed steps are given in Algorithm~\ref{alo1}.
\begin{algorithm}
	\caption{Reconstruct [$h$] in $\mathrm{PW}_B^p$ by the square-root method.}\label{alo1}
	\begin{algorithmic}[1]
		\STATE  Suppose $b$ satisfies the hypotheses of Proposition \ref{prop3} and $\left\{\omega_{n}=n\mathfrak{b}\right\}$, where $0<\mathfrak{b}\le 1/(4B)$;
		\STATE given $|h(\omega_{n})|$, $|h(\omega_{n}+b)-h(\omega_{n})|$ for $n \in \mathbb{Z}$;
		\STATE use \eqref{equ1} or \eqref{equ2.2} to calculate $|h(\cdot)|^{2}$ and $|h(\cdot+b)-h(\cdot)|^{2}$;
		\STATE compute
		\begin{align}\label{equ3-1}
			R(x)=\frac{|h(x+b)|^{2}+|h(x)|^{2}-|h(x+b)-h(x)|^{2}}{2};
		\end{align}
		\STATE compute
		\begin{align}\label{equ3-2}
			Q(x)=|h(x+b)|^{2}|h(x)|^{2}-|R(x)|^{2};
		\end{align}
		\STATE choose a continuous square root $I$ of $Q$ such that $I(x)^2=Q(x)$. The global sign is arbitrary; $-I$ is the only alternative;
		\STATE construct
		\begin{align}\label{equ3-3}
			\widetilde K(x)=R(x)+iI(x);
		\end{align}
		\STATE choose the appropriate $\beta \in [0, b)$, set $x_n=\beta+nb$, $h_n=h(x_n)$ and $h_0=\sqrt{|h(\beta)|^2}$;
		\STATE recover the remaining samples outward via
		\begin{align}\label{equ3-4}
			h_n =
			\begin{cases}
				\dfrac{\widetilde K(x_{n-1})}{\overline{h_{n-1}}}, & n > 0, \\
				\dfrac{\overline{\widetilde K(x_{n})}}{\overline{h_{n+1}}}, & n < 0;
			\end{cases}
		\end{align}
		
		\STATE apply \eqref{equ1} or \eqref{equ2.2} to reconstruct either $\lambda h$ or  $\lambda \overline{h}$ (and hence $[h]$) from the values $h_n$.
	\end{algorithmic}
\end{algorithm}

Let $h\in \mathrm{PW}_B^p$ be a complex-valued bandlimited signal. When
$p=\infty$, we additionally assume that $\hat h\in L^1(\mathbb R)$. Suppose we are given the magnitude measurements
\[
\left\{|h(n\mathfrak{b})|, |h(n\mathfrak{b}+b)-h(n\mathfrak{b})|: n\in \mathbb{Z}\right\}
\]
where $b$ satisfy the hypotheses of Proposition \ref{prop3} and $0<\mathfrak{b}\le 1/(4B)$. Our goal is to recover $h$ up to a global phase and conjugation.

First, since $h \in \mathrm{PW}_B^p$, both $|h|^2$ and $|h(\cdot+b)-h(\cdot)|^2$ belong to $\mathrm{PW}_{2B}^p\subset \mathrm{PW}_{1/(2\mathfrak{b})}^p$ with $0<\mathfrak{b}\le 1/(4B)$. Hence, applying the interpolation formula \eqref{equ1} for $1\le p<\infty$ or \eqref{equ2.2} for $p=\infty$ to each of these two functions, we reconstruct the continuous functions
\[
|h(\cdot)|^2 \quad\text{and}\quad |h(\cdot+b)-h(\cdot)|^2
\]
from their samples on the sequence $\{n\mathfrak{b}\}$. Next, define the adjacent correlation function
\begin{align*}
	K(x): = h(x+b)\overline{h(x)}.
\end{align*}
By the cosine law in the complex plane,
\[
|h(x+b) - h(x)|^2 = |h(x+b)|^2 + |h(x)|^2 - 2 \operatorname{Re} K(x),
\]
which yields the real part $R(x):=\operatorname{Re} K(x)$ from \eqref{equ3-1}. Moreover, since \[
|K(x)|^2 = |h(x+b)|^2 |h(x)|^2,
\]
the squared imaginary part $Q(x):=(\operatorname{Im} K(x))^2$ is computed via \eqref{equ3-2}.

Since $h \in \mathrm{PW}_B^p$, the function $K(x)$ is the product of two bandlimited functions and hence belongs to $\mathrm{PW}_{2B}^p$. In particular, both $K$ and $\operatorname{Im} K$ are real-analytic on $\mathbb{R}$. Consequently,  $Q=(\operatorname{Im} K)^2$ admits a continuous square root $I$ satisfying  $I(x)^2=Q(x)$ on $\mathbb{R}$. The function  $I$ is determined uniquely up to a global sign. Indeed, if $Q \equiv 0$, we set $I \equiv 0$. Otherwise, the zeros of $Q$ are isolated and of even multiplicity. On each connected component of the open set $\{Q>0\}$, the square root has exactly two continuous choices, namely $\pm\sqrt{Q}$. Continuity across any zero forces the signs on adjacent components to coincide, since $I$ must vanish at the zero itself. Hence, once the sign is chosen on any one component, it is uniquely determined on all components, leaving only the global sign ambiguity $I\mapsto -I$.

Then we construct the candidate $\widetilde K$ from \eqref{equ3-3}. The alternative $R(x)-iI(x)$ is the complex conjugate of $\widetilde K$. Exactly one of the two equals the true correlation function $K(x)$, while the other equals its conjugate $\overline{K(x)}$.

For the given candidate $\widetilde K$, we recover a sample sequence by outward recurrence. Choose an offset $\beta\in[0,b)$ such that the shifted grid $x_n=\beta+nb$ satisfies $$|h(x_n)|^{2}>0,\quad\forall n\in\mathbb{Z}.$$ Such a $\beta$ always exists unless $h\equiv0$, because the zeros of a nonzero bandlimited function form a discrete set. Set $h_0=\sqrt{|h(\beta)|^{2}}>0$ to fix the global phase ambiguity. The remaining samples are obtained from the relation $K(x_n)=h_{n+1}\overline{h_n}$ via \eqref{equ3-4}.

\subsubsection{K-Sign Method.}
The square-root method presented in Algorithm~\ref{alo1} relies critically on the extraction of an analytic square root $I(x)$ of $Q(x)$, which is determined only up to a global sign. However, in the presence of noise, small perturbations in the measured magnitudes can induce local sign flips in the computed square root near the zeros of $Q$. These isolated errors are then propagated through the outward recurrence, causing large-scale phase distortions that may render the reconstruction useless. To overcome this instability, we introduce an additional measurement, which provides redundant information that enables us to determine the sign of $\operatorname{Im}K(x)$ locally, rather than globally. We refer to this robust procedure as the $K$-sign method, described in Algorithm~\ref{alo7}.
\begin{algorithm}
	\caption{Reconstruct [$h$] in $\mathrm{PW}_B^p$ by the $K$-sign method.}\label{alo7}
	\begin{algorithmic}[1]
		\STATE Suppose $b$ satisfies the hypotheses of Proposition \ref{prop3} and $\left\{\omega_{n}=n\mathfrak{b}\right\}$, where $0<\mathfrak{b}\le 1/(4B)$. Define $b_1=b$, $b_2=2b$;
		\STATE given $|h(\omega_{n})|$, $|h(\omega_{n}+b_k)-h(\omega_{n})|$ for $n \in \mathbb{Z}$ and $k=1,2$;
		\STATE use \eqref{equ1} or \eqref{equ2.2} to calculate $|h(\cdot)|^{2}$ and $|h(\cdot+b_k)-h(\cdot)|^{2}$ for $k=1,2$;
		\STATE compute
		\[
		R_{k}(x)=\frac{|h(x+b_k)|^{2}+|h(x)|^{2}-|h(x+b_k)-h(x)|^{2}}{2},\quad k=1,2;
		\]
		\STATE compute
		\begin{align*}
			Q(x)=|h(x+b_1)|^{2}|h(x)|^{2}-|R_1(x)|^{2};
		\end{align*}
		\STATE choose the appropriate $\beta\in[0,b_1)$, set $x_n=\beta+nb_1$, $h_n=h(x_n)$ and $h_0=\sqrt{|h(\beta)|^2}$;
		\STATE compute
		\[
		\rho_{n}=\mathop{\rm argmin}_{\rho\in\{-1,1\}}
		\left|R_{1}(x_{n+1})R_{1}(x_n)-|h(x_{n+1})|^2R_{2}(x_n)
		-\rho S(x_{n+1})S(x_n)\right|,
		\]
		where $S(x_n)=\sqrt{Q(x_n)}$;
		\STATE if $Q\not\equiv0$,
		set $\varepsilon_{0}=1$ and determine the remaining signs outward from
		$\varepsilon_{n+1}\varepsilon_{n}=\rho_{n}$,
		if $Q\equiv0$, set $\varepsilon_{n}=0$;
		\STATE construct
		\begin{align}\label{eq:con}
			\widetilde K(x_n)=R_{1}(x_n)+i\varepsilon_{n}S(x_n);
		\end{align}
		\STATE recover the remaining samples outward via \eqref{equ3-4};
		\STATE apply \eqref{equ1} or \eqref{equ2.2} to reconstruct either $\lambda h$ or  $\lambda \overline{h}$ (and hence $[h]$) from the values $h_n$.
	\end{algorithmic}
\end{algorithm}

Let $b_1=b$ and $b_2=2b$. Suppose we are given
\[
\{|h(\omega_{n})|, |h(\omega_{n}+b_k)-h(\omega_{n})|: n\in \mathbb{Z}, k=1,2\}.
\]
The structure is almost the same as that of Algorithm~\ref{alo1}, with the only difference being the procedure for choosing the signs. For $k=1,2$, let
\[
K_{k}(x)=h(x+b_k)\overline{h(x)}.
\]
From the definitions in Steps~4--5, we have
\[
R_k(x)=\operatorname{Re}K_k(x),\quad Q(x)=|\operatorname{Im}K_1(x)|^2.
\]

In Algorithm~\ref{alo1}, a continuous square root $I(x)$ is chosen globally, and the ambiguity is simply the global sign  $I\mapsto -I$. The $K$-sign method, by contrast, determines the sign of $\operatorname{Im}K_1$ at each grid point individually. This local determination is made possible by the third measurement, which provides the additional correlation function $K_2$.

The key identity underlying the sign selection is
\begin{align}\label{eq:ksign}
	K_1(x+b_1)K_1(x)&=h(x+2b_1)\overline{h(x+b_1)}h(x+b_1)\overline{h(x)}\notag
	\\&=|h(x+b_1)|^2K_2(x),
\end{align}
which follows directly from $b_2=2b_1 $.

If $Q\equiv 0$, then $\operatorname{Im}K_1\equiv 0$, i.e., $K_1$ is real-valued; in this degenerate case no sign selection is required, and the offset $\beta$ in Step~6 is chosen only to satisfy $|h(x_n)|^2>0$ for all $n\in\mathbb Z$. Our primary concern is the sign selection in the nondegenerate case $Q\not\equiv 0$, the offset $\beta$ is chosen so that
\[
|h(x_n)|^2>0\quad\hbox{and}\quad Q(x_n)>0,\qquad \forall n\in\mathbb Z.
\] Such a choice is feasible because the zeros of a nonzero bandlimited function, and the zeros of the nonzero real-analytic function $\operatorname{Im}K_1$, are discrete.

Evaluating the real part of \eqref{eq:ksign} at the grid points $x_n=\beta+n b_1$ yields
\begin{align}\label{eq:re}
	R_1(x_{n+1})R_1(x_n)-\bigl(\operatorname{Im}K_1(x_{n+1})\bigr)\bigl(\operatorname{Im}K_1(x_n)\bigr)
	=|h(x_{n+1})|^2R_2(x_n) .
\end{align}
Now write $\operatorname{Im}K_1(x_n)=\varepsilon_n S(x_n) $, where $\varepsilon_n\in\{-1,1\} $. Then from \eqref{eq:re}, we have
\begin{align}\label{eq:re1}
	\varepsilon_{n+1}\varepsilon_nS(x_{n+1})S(x_n)
	= R_1(x_{n+1})R_1(x_n)-|h(x_{n+1})|^2R_2(x_n).
\end{align}

Therefore, the third measurement determines whether the two adjacent imaginary
parts have the same sign or opposite signs. Once $\varepsilon_0$ is fixed (we set  $\varepsilon_0=1$ without loss of generality, as this only selects one of the two global conjugate solutions), all successive signs are obtained by the simple recurrence
\[
\varepsilon_{n+1} = \rho_n \varepsilon_{n}.
\]
Thus the local sign of $\operatorname{Im}K_1$ is resolved everywhere on the grid.

With the reconstructed correlation values \eqref{eq:con}, we have either $\widetilde K(x_n)=K_1(x_n)$ or $\widetilde K(x_n)=\overline{K_1(x_n)}$ (global conjugation). The recurrence \eqref{equ3-4} then recovers the samples $h_n$ up to a global phase and conjugation, and the Shannon formula yields the equivalence class $[h]$.

\subsection{STFT Phase Retrieval}
With the reconstruction methods for $\mathrm{PW}_B^p$ established above, we now incorporate them into the STFT framework.
\subsubsection{Two-Window STFT via the Square-Root Method}
We first consider the two-window scheme, in which the per-slice reconstruction is performed by the square-root method of Algorithm~\ref{alo1}. The complete procedure is summarized in Algorithm~\ref{alo2}.
\begin{algorithm}
	\caption{Two-Window STFT via the Square-Root Method.}\label{alo2}
	%\label{algorithmlabel}
	\begin{algorithmic}[1]
		\STATE Suppose $\phi, a, b$ satisfy the hypotheses of Proposition \ref{prop1} and $\left\{\omega_{n}=n\mathfrak{b}\right\}$, where $0<\mathfrak{b}\le 1/(4B)$;
		\STATE given $|\mathcal{V}_{\phi} f(ma,\omega_{n})|$, $|\mathcal{V}_{\psi} f(ma,\omega_{n})|$ for $m,n \in \mathbb{Z}$;
		\STATE for $m\in \mathbb{Z}$ is arbitrary but fixed, use Algorithm \ref{alo1} to reconstruct either $\lambda_{m} h_{m}$ or  $\lambda_{m} \overline{h_{m}}$ (and hence $[h_{m}]$) for our choice of uniform phase factor $\lambda_{m}$ and conjugation, denoted as $\widetilde{h_{m}}$;
		\STATE calculate
		\begin{align}\label{equ3-5}
			F_{m}(x)=\mathcal{F}\widetilde{h_{m}}(x),\quad m \in \mathbb{Z};
		\end{align}
		\STATE for $m$, calculate
		\begin{align}\label{equ3-6}
			f_m(x)=\frac{\phi(x)}{|\phi(x)|^2}\overline{F_{m}(x)},\quad x\in (-B,B);
		\end{align}
		\STATE for $m$, choose the phase and conjugate reflection of $f_m$ so that
		\begin{align*}
			f_{m} (x)=f_{m+1} (x-a),\quad x \in (-B+a,B),
		\end{align*}
		hence, reconstruct $\lambda f$ from the selected consistent phase factor and the conjugate reflection of $f_{m}$.
	\end{algorithmic}
\end{algorithm}

Suppose $\phi, a, b$ satisfy the hypotheses of Proposition \ref{prop1} and $\left\{\omega_{n}=n\mathfrak{b}\right\}$, where $0<\mathfrak{b}\le 1/(4B)$. Let $\psi$ be defined by
\begin{align}\label{equ3-7}
	\psi(x)=\phi(x)(e^{2 \pi ixb}-1).
\end{align}

The first strategy of Algorithm \ref{alo2} is to notice that
\begin{align*}
	h_{m}(\omega):=M_{ma}\mathcal{V}_{\phi} f(ma,\omega)=\mathcal{F}\left(f(\cdot+ma) \overline{\phi(\cdot)}\right)(\omega),
\end{align*}
which belongs to $\mathrm{PW}_B^p $. Moreover, $\widehat h_m\in L^1(\mathbb R)$, so the additional
assumption in Proposition~\ref{prop2.3} is automatically satisfied. A direct computation shows that
\begin{align}\label{equ3-8}
	|h_{m}(\omega+b)-h_{m}(\omega)|=|\mathcal{V}_{\psi} f(ma,\omega)|.
\end{align}
Therefore, from the STFT magnitude measurements
\begin{align*}
	\{|\mathcal{V}_{\phi} f(ma,\omega_{n})|, |\mathcal{V}_{\psi} f(ma,\omega_{n})|:~m,n\in\mathbb{Z}\},
\end{align*}
we obtain the estimate of $[h_{m}]$ for each $m$ via Algorithm \ref{alo1}, which is denoted by $\widetilde{h_{m}}$. In other words, there exists some $\lambda_{m}\in \mathbb{T}$ such that
\begin{align}\label{equ3-9}
	\widetilde{h_{m}}=\lambda_{m} h_{m}~\text{or}~\lambda_{m} \overline{h_{m}},\quad m\in\mathbb{Z}.
\end{align}

Substitute \eqref{equ3-9} into \eqref{equ3-5}, we obtain that
\begin{align*}
	F_{m}(x)=\lambda_{m}\mathcal{F}h_{m}(x)~\text{or}~
	\lambda_{m}\overline{\mathcal{F}h_{m}(-x)}
	,\quad \mbox{ a.e. } x\in\mathbb{R},~m\in\mathbb{Z}.
\end{align*}
Combined with \eqref{equ3-8}, we have
\begin{align*}
	F_{m}(x)=\lambda_{m}f(-x+ma)\phi(x)~\text{or}~\lambda_{m}\overline{f(x+ma)}\phi(x),\quad \mbox{ a.e. } x\in (-B,B),~m\in\mathbb{Z}.
\end{align*}
Then by \eqref{equ3-6}, it can be obtained that
\begin{align}\label{equ3-10}
	f(x+ma)=\lambda_{m}f_{m}(x)~\text{or}~\overline{\lambda_{m}f_{m}(-x)},\quad \mbox{ a.e. } x\in (-B,B),~m\in\mathbb{Z}.
\end{align}
To resolve these ambiguities globally, we exploit the overlap between consecutive time slices. Indeed, for $m$ and $m+1 $, we note that the domains of $f$ corresponding to $f_{m}$ and $f_{m+1}$ overlap on the interval $(-B+(m+1)a, B+ma)$. On this overlap, the two representations must agree, up to the possible phase and conjugation choices. By the above fact and \eqref{equ3-10}, it can be obtained that there may be two cases:
\\
Case 1: If
\begin{align*}
	f_{m}(x)=\overline{\lambda_{m}}\lambda_{m+1}f_{m+1}(x-a)~\text{or}~\overline{\lambda_{m}\lambda_{m+1}f_{m+1}(-(x-a))},\quad \mbox{ a.e. } x\in (-B+a,B)
\end{align*}
holds, we need to rectify the ambiguity of phase factor and choice of conjugate reflection of $f_{m+1}$ so that $f_{m}(x)=f_{m+1}(x-a)$ on $(-B+a,B)$.
\\
Case 2: If
\begin{align*}
	f_{m}(-x)=\overline{\lambda_{m}}\lambda_{m+1}f_{m+1}(-(x-a))~\text{or}~\overline{\lambda_{m}\lambda_{m+1}f_{m+1}(x-a)},\quad \mbox{ a.e. } x\in (-B+a,B)
\end{align*}
holds, we first rectify the reflection of $f_{m}$, then rectify the ambiguity of phase factor and choice of conjugate reflection for $f_{m+1}$ to ensure that $f_{m}(x)=f_{m+1}(x-a)$ holds on $(-B+a,B)$.

Therefore, it is noted that in step 6, the choice of phase factor can be arbitrarily initialized for $f_{0}$, and then propagate the corrections outward to $m>0$ and $m<0$ sequentially. This yields a globally consistent representation  $\lambda f$ for some $\lambda\in\mathbb T$, thus completing the reconstruction.

\subsubsection{Three-Window STFT via the K-Sign Method}
To improve robustness against noise, we replace the square-root step with the $K$-sign method and introduce a third window. And the overall STFT structure remains unchanged. The resulting algorithm is summarized in Algorithm~\ref{alo8}.

\begin{algorithm}
	\caption{Three-Window STFT via the $K$-Sign Method.}\label{alo8}
	%\label{algorithmlabel}
	\begin{algorithmic}[1]
		\STATE Suppose $\phi, a, b$ satisfy the hypotheses of Proposition \ref{prop1} and $\left\{\omega_{n}=n\mathfrak{b}\right\}$, where $0<\mathfrak{b}\le 1/(4B)$;
		\STATE given $|\mathcal{V}_{\phi} f(ma,\omega_{n})|$, $|\mathcal{V}_{\psi_k} f(ma,\omega_{n})|$ for $m,n \in \mathbb{Z}$ and $k=1,2$;
		\STATE for $m\in \mathbb{Z}$ being arbitrary but fixed, use $K$-sign Algorithm to reconstruct either $\lambda_{m} h_{m}$ or  $\lambda_{m} \overline{h_{m}}$ (and hence $[h_{m}]$) for our choice of uniform phase factor $\lambda_{m}$ and conjugation, denoted as $\widetilde{h_{m}}$;
		\STATE Subsequent steps are identical to Steps~4--6 of Algorithm~\ref{alo2}.
	\end{algorithmic}
\end{algorithm}

For each time slice $m$, define $h_m$ as in the two-window case. Using three windows $\phi,\psi_1,\psi_2$ with
\[
\psi_k(x)=\phi(x)\bigl(e^{2\pi i b_k x}-1\bigr),\qquad b_1=b,\ b_2=2b,
\]
we obtain the measurements
\[
\{|\mathcal V_\phi f(ma,\omega_n)|,~|\mathcal V_{\psi_k} f(ma,\omega_n)|:\ m,n\in\mathbb Z,\ k=1,2\}.
\]

For fixed $m$, the quantities $|h_m(\omega_n)|$ and $|h_m(\omega_n+b_k)-h_m(\omega_n)|$ can be extracted from these STFT magnitudes exactly as in the two-window case (cf. \eqref{equ3-8}). Applying Algorithm~\ref{alo7} to each $m$ yields a representative $\widetilde h_m$ of the equivalence class $[h_m]$. The subsequent phase and conjugation alignment across overlapping time slices is performed identically to Steps~4--6 of Algorithm~\ref{alo2}.
\begin{remark}
	
	As an extension of the preceding $K$-sign method, we note that its core procedure, namely the conjugate phase retrieval in $\mathrm{PW}_B^p $, can be replaced by the structured convolution reconstruction (SCR) algorithm of Lai, Littmann, and Weber~\cite{2021Conjugate}.
	
	The SCR algorithm of
Lai,   Littmann, and  Weber
\cite{2021Conjugate} recovers $f\in \mathrm{PW}_B^p$ from magnitudes of structured convolutions $|\vec{v}_l * f(\omega_n)|$, where $\{\vec{v}_l\}$ forms a matrix performing conjugate phase retrieval in a finite-dimensional space. For $\mathbb{C}^3 $, the matrix
	\begin{align}\label{equ5.2}
		V=\left[\begin{array}{rrrrrr}
			1 & 0 & 0 & 1 & 1 & 0 \\
			0 & 1 & 0 & -1 & 0 & 1 \\
			0 & 0 & 1 & 0 & -1 & -1
		\end{array}\right]
	\end{align}
	provides a concrete example. Embedding this method into the STFT framework with $b_0=0$,  $b_1=b $, $b_2=2b $, where $0<b\le 1/(2B)$, yields Algorithm~\ref{alo4}.
\end{remark}

\begin{algorithm}
	\caption{Three-Window STFT via SCR.}
	\label{alo4}
	\begin{algorithmic}[1]
		\STATE Suppose $\phi, a, b$ satisfy the hypotheses of Proposition \ref{prop1} and $\left\{\omega_{n}=n\mathfrak{b}\right\}$, where $0<\mathfrak{b}\le 1/(4B)$, set $b_0=0$, $b_1=b$, $b_2=2b$ and choose $V$ as in \eqref{equ5.2};
		\STATE given $|\mathcal{V}_{\phi} f(ma,\omega_{n})|$, $|\mathcal{V}_{\psi_k} f(ma,\omega_{n})|$ for $m,n \in \mathbb{Z}$ and $k=1,2$;
		\STATE for each fixed $m\in\mathbb Z $, apply the SCR algorithm of \cite{2021Conjugate} to obtain $\widetilde h_m$;
		\STATE Resolve per-slice ambiguities via overlap correction (Steps~4--6 of Algorithm~\ref{alo2}).
	\end{algorithmic}
\end{algorithm}

\section{Deterministic Stability on a Fixed Time Interval}
	
	We consider the stability analysis under a fixed time interval.
	Fix $m\in\mathbb Z$ and set
	\begin{align}\label{eq:wd}
		g(t):=f(t+ma)\overline{\phi(t)},
		\qquad
		h(x):=\mathcal Fg(x).
	\end{align}
	Since Algorithm~\ref{alo7} is formulated for prescribed STFT sampling sites, it must first use the
	Shannon formula to reconstruct the three continuous squared-magnitude
	functions and then select a suitable sequence $\beta+b\mathbb Z$.  In the
	stability analysis below, however,
	$\beta$ is fixed at the initial stage when the STFT measurements are given.
	More precisely, the observed data consist of a STFT sequence
	$\{ |\mathcal{V}_\phi f(ma,\beta+nb)| \}_n$,
	and the sequence $\{ \beta+nb \}_n$ is precisely the sequence $\{x_n\}$ defined in Step~6 of Algorithm~\ref{alo7}.
	The noise is therefore introduced directly into this fixed STFT sequence.

	We assume throughout this section that $g\in L^2(\mathbb R)$ and
	$g\ne0$.
	Then $\operatorname{supp}g\subset[-B,B]$ and
	$h\in\mathrm{PW}_B^2$.
	Fix $\beta\in[0,b)$.  Define
	\[
	x_n=\beta+nb,
	\quad h_n=h(x_n),
	\quad n\in\mathbb Z.
	\]
	In the notation of Algorithm~\ref{alo7}, the exact data on this sequence are
		\begin{align*}
		r_n&=|h_n|=|\mathcal{V}_\phi f(ma,x_n)|, \\
		d_n&=|h_{n+1}-h_n|=|\mathcal{V}_{\psi_1} f(ma,x_n)|, \\
		\ell_n&=|h_{n+2}-h_n|=|\mathcal{V}_{\psi_2} f(ma,x_n)|.
	\end{align*}
	The measured data are
	\[
	\tilde{r}_n=r_n+e_n^0, \quad
	\tilde{d}_n=d_n+e_n^1, \quad
	\tilde{\ell}_n=\ell_n+e_n^2,
	\]
	where all three measured magnitudes are nonnegative.
	Define
	\[
	K_n:=h_{n+1}\overline{h_n}.
	\]
	Then
	\[
	R_n=\operatorname{Re} K_n=\frac{r_{n+1}^2+r_n^2-d_n^2}{2},
	\qquad
	Q_n=|\operatorname{Im} K_n|^2=r_{n+1}^2r_n^2-R_n^2.
	\]
	Set $S_n=|\operatorname{Im} K_n|$. The third sequence is used only to determine the relative signs of consecutive imaginary parts.  Indeed, define
	\[
C_n:=\frac{r_{n+2}^2+r_n^2-\ell_n^2}{2}
=\operatorname{Re}(h_{n+2}\overline{h_n}).
\]
	and by \eqref{eq:re}, we have
	\begin{equation}\label{eq:stab-G}
	G_n:=R_{n+1}R_n-r_{n+1}^2C_n=\operatorname{Im} K_{n+1}\operatorname{Im} K_n.
    \end{equation}

	To avoid imposing a nondegeneracy condition on arbitrarily small tail
	samples, we restrict the branch analysis to a finite set containing most of
	the sampled energy and include its complement in the error.  Parseval gives
	\begin{align}\label{eq:Eall}
		E_{\rm all}:=\sum_{n\in\mathbb Z}r_n^2
		=\frac{\|g\|_{L^2}^2}{b}>0.
	\end{align}
	Choose a finite contiguous block of frequency samples. After translating the integer
	index, write it as
	\begin{align}\label{eq:EJ}
		J=\{0,1,\ldots,N\},
		\qquad
		E_J:=\sum_{n=0}^{N}r_n^2>0.
	\end{align}
	Its energy fraction is
	\begin{align}\label{eq:AJ}
		\alpha_J:=\frac{E_J}{E_{\rm all}}
		=\frac{\sum_{n\in J}|h_n|^2}
		{\sum_{n\in\mathbb Z}|h_n|^2}.
	\end{align}
	Thus, for example, $\alpha_J\geq0.9$ means that $J$ contains at least $90\%$ of the total sampled energy.

	Assume that the noise restricted to $J$ satisfies
	\begin{equation}\label{eq:stab-noise}
	\sum_{n=0}^{N}|e_n^0|^2
	+\sum_{n=0}^{N-1}|e_n^1|^2
	+\sum_{n=0}^{N-2}|e_n^2|^2
	\leq\delta^2E_J.
\end{equation}
	Thus, $\delta$ measures the combined
	$\ell^2$-noise on $J$ relative to the
	$\ell^2$-norm $E_J^{1/2}$ of the exact magnitude samples.
	Moreover, we assume the pointwise bound
\begin{equation}
	\max\left\{
	\max_{0\leq n\leq N}|e_n^0|,
	\max_{0\leq n\leq N-1}|e_n^1|,
	\max_{0\leq n\leq N-2}|e_n^2|
	\right\}
	\leq\eta\sqrt{E_J}.
	\label{eq:stab-noise-max}
\end{equation}
The parameter $\eta$ will be used only for the pointwise square-root and
branch decisions.

From the noisy data, set
\begin{align*}
	\widetilde R_n
	&=\frac{\widetilde r_{n+1}^2+\widetilde r_n^2
		-\widetilde d_n^2}{2},
	&
	\widetilde Q_n
	&=\widetilde r_{n+1}^2\widetilde r_n^2
	-\widetilde R_n^2,\\
	\widetilde C_n
	&=\frac{\widetilde r_{n+2}^2+\widetilde r_n^2
		-\widetilde\ell_n^2}{2},
	&
	\widetilde G_n
	&=\widetilde R_{n+1}\widetilde R_n
	-\widetilde r_{n+1}^2\widetilde C_n.
\end{align*}
Whenever $\widetilde Q_n>0$, define
\[
\widetilde S_n=\sqrt{\widetilde Q_n}.
\]
Condition \eqref{eq:stab-smallness} below guarantees that
$\widetilde Q_n>0$ for every $0\leq n\leq N-1$.
Choose $\widetilde\varepsilon_0\in\{-1,1\}$ arbitrarily. For
$0\leq j\leq N-2$, choose $\widetilde\varepsilon_{j+1}\in\{-1,1\}$ recursively by the rule from Algorithm~\ref{alo7}, i.e.,
\[
\widetilde\varepsilon_{j+1}
=\mathop{\rm argmin}_{\varepsilon\in\{-1,1\}}
\left|
\widetilde G_{j}
-\varepsilon\widetilde\varepsilon_{j}\widetilde S_{j+1}\widetilde S_{j}\right|.
\]
Set $\widetilde K_n=\widetilde R_n
+i\widetilde\varepsilon_n\widetilde S_n$ for $0\leq n\leq N-1$,
and
\[
\widetilde h_0=\widetilde r_0,
\qquad
\widetilde h_n=
\frac{\widetilde K_{n-1}}
{\overline{\widetilde h_{n-1}}},
\quad 1\leq n\leq N.
\]
For $\sigma\in\{-1,1\}$, write
\[
h_n^\sigma=
\begin{cases}
	h_n,&\sigma=1,\\
	\overline{h_n},&\sigma=-1.
\end{cases}
\]

	\begin{theorem}[Stability on the frequency block]
		\label{thm:stab-core}
		Assume that the noise restricted to $J$ satisfies \eqref{eq:stab-noise} and \eqref{eq:stab-noise-max}, and suppose that
		for some $\nu>0$,
		\begin{equation}
	S_n\geq\nu E_J,
	\qquad 0\leq n\leq N-1.
	\label{eq:stab-transversality}
\end{equation}
Assume also that the branches are correct on $J$, namely, for one
$\sigma\in\{-1,1\}$,
\begin{equation}
	\widetilde\varepsilon_n
	=\sigma\operatorname{sgn}(\operatorname{Im} K_n),
	\qquad 0\leq n\leq N-1.
	\label{eq:stab-branch}
\end{equation}
If
\begin{equation}
	A_\eta:=\left(2\sqrt2+\sqrt{1+\sqrt{1-4\nu^2}}\right)\eta
	+2\eta^2<\frac{1-\sqrt{1-4\nu^2}}{2},
	\label{eq:stab-smallness}
\end{equation}
then $\widetilde Q_n>0$ for every $0\leq n\leq N-1$, the above reconstruction
is well defined, and there exists some
$\lambda\in\mathbb T$ such that
\begin{equation}
	\frac{1}{E_J}\sum_{n=0}^{N}
	|\widetilde h_n-\lambda h_n^\sigma|^2
	\leq
	\delta^2\left(
	1+\frac{5(1+\delta)(1+\delta/4)^2}
	{\nu^2-A_\eta+A_\eta^2}
	\right).
	\label{eq:stab-core-bound}
\end{equation}

Moreover, with $\varepsilon_\eta=2\sqrt2\,\eta+\eta^2$, the branch
condition \eqref{eq:stab-branch} follows from the third measured sequence if
\begin{equation}
	\frac72\varepsilon_\eta
	+\frac{15}{4}\varepsilon_\eta^2<\nu^2.
	\label{eq:stab-zero-branch-condition}
\end{equation}
In this case, either choice of $\widetilde\varepsilon_0$ gives one of the two
global conjugate branches, and \eqref{eq:stab-core-bound} holds.
\end{theorem}	

\begin{proof}
We prove the conclusion in three steps.

(i)\, For $0\leq t\leq1$, let
\[
r_n(t)=r_n+te_n^0,
\qquad d_n(t)=d_n+te_n^1,
\]
and define
\[
R_n(t)=\frac{r_{n+1}^2(t)+r_n^2(t)-d_n^2(t)}{2},
\qquad
q_n^\pm(t)=r_{n+1}(t)r_n(t)\pm R_n(t).
\]
At $t=0$, by \eqref{eq:stab-transversality}, we have
\begin{equation}\label{eq:qn0}
	q_n^+(0)q_n^-(0)=Q_n=|S_n|^2\geq\nu^2E_J^2.
\end{equation}
Since $$q_n^+(0)+q_n^-(0)=2r_{n+1}r_n\leq r_{n+1}^2+r_{n}^2\leq E_J,$$
letting $x$ denote either of the two numbers $q_n^+(0)$ or $q_n^-(0)$,
we obtain
\begin{equation}\label{eq:qn1}
	\nu^2E_J^2\le q_n^+(0)q_n^-(0)\le x(E_J-x).
\end{equation}
Therefore
\begin{equation}\label{equ:qn0}
	q_n^\pm(0)\geq \frac{1-\sqrt{1-4\nu^2}}{2}E_J.
\end{equation}
Here the condition $0<\nu\le 1/2$ follows immediately from
\[
\nu E_J\le S_n\le |K_n|= r_{n+1}r_n\le \frac{E_J}{2}.
\]
A direct expansion of $q_n^+(t)$ yields
\begin{equation}\label{equ qn}
	q_n^+(t)-q_n^+(0)=t(r_{n+1}+r_n)(e_{n+1}^0+e_n^0)-td_ne_n^1+\frac{t^2}{2}\big((e_{n+1}^0+e_n^0)^2-(e_n^1)^2\big).
\end{equation}
From the elementary inequality
\[
(r_{n+1}+r_n)^2\le 2(r_{n+1}^2+r_n^2)\le 2E_J,
\]
we obtain
\begin{align*}
	d_n^2&=r_{n+1}^2+r_n^2-2R_n=(r_{n+1}+r_n)^2-2q_n^+(0)\\
	&\le 2E_J-(1-\sqrt{1-4\nu^2})E_J=(1+\sqrt{1-4\nu^2})E_J.
\end{align*}
Hence
\begin{align}
	|(r_{n+1}+r_n)(e_{n+1}^0+e_n^0)-d_ne_n^1|
	&\leq(r_{n+1}+r_n)(|e_{n+1}^0|+|e_n^0|)+d_n|e_n^1|\nonumber\\
	&\leq\left(2\sqrt2+\sqrt{1+\sqrt{1-4\nu^2}}\right)\eta E_J.\label{equ t 1}
\end{align}
Notice that
\begin{align}
	\frac12\left|(e_{n+1}^0+e_n^0)^2-(e_n^1)^2\right|
	&\leq \frac12\max\left\{(e_{n+1}^0+e_n^0)^2,(e_n^1)^2\right\}\nonumber\\
	&\leq 2\eta^2E_J.\label{equ noise}
\end{align}
Applying \eqref{equ t 1} and \eqref{equ noise} to the identity \eqref{equ qn}, we obtain
\begin{equation}
	\frac{\left|q_n^+(t)-q_n^+(0)\right|}{E_J}
	\leq
	\left(2\sqrt2+\sqrt{1+\sqrt{1-4\nu^2}}\right)t\eta
	+2t^2\eta^2.
	\label{eq:stab-qplus-perturb}
\end{equation}
Similarly, we can also obtain
\begin{equation}
	\frac{\left|q_n^-(t)-q_n^-(0)\right|}{E_J}
	\leq
	\left(2\sqrt2+\sqrt{1+\sqrt{1-4\nu^2}}\right)t\eta
	+2t^2\eta^2.
	\label{eq:stab-qminus-perturb}
\end{equation}
Set
\[
\mu(t):=\left(2\sqrt2+\sqrt{1+\sqrt{1-4\nu^2}}\right)t\eta
+2t^2\eta^2.
\]
Since $0\leq t\leq1$, we have $0\leq\mu(t)\leq A_\eta$. Applying \eqref{eq:stab-qplus-perturb} and \eqref{eq:stab-qminus-perturb} for the two signs,  we obtain the unified bound
\begin{equation}
	|q_n^\pm(t)-q_n^\pm(0)|\leq\mu(t)E_J\leq A_\eta E_J.
	\label{eq:stab-q-perturb}
\end{equation}
By \eqref{equ:qn0} and \eqref{eq:stab-q-perturb}, we deduce that
\[
q_n^\pm(t)\geq q_n^\pm(0)-|q_n^\pm(t)-q_n^\pm(0)|\geq
\left(\frac{1-\sqrt{1-4\nu^2}}{2}-\mu(t)\right)E_J.
\]
Combining this estimate with the smallness condition \eqref{eq:stab-smallness}, we deduce that the last quantity is strictly positive.
Moreover, combining \eqref{eq:qn0} and \eqref{eq:qn1} gives
\begin{align}
	q_n^+(t)q_n^-(t)
	&\geq\bigl(q_n^+(0)-\mu(t)E_J\bigr)
	\bigl(q_n^-(0)-\mu(t)E_J\bigr)\nonumber\\
	&=q_n^+(0)q_n^-(0)
	-\mu(t)E_J\bigl(q_n^+(0)+q_n^-(0)\bigr)
	+\mu^2(t)E_J^2\nonumber\\
	&\geq\bigl(\nu^2-\mu(t)+\mu^2(t)\bigr)E_J^2\nonumber\\
	&\geq\bigl(\nu^2-A_\eta+A_\eta^2\bigr)E_J^2.\label{eq:stab-I-path}
\end{align}
The penultimate step uses the fact that
\[
0\leq\mu(t)\leq A_\eta
<\frac{1-\sqrt{1-4\nu^2}}2\leq\frac12,
\]
while the strict positivity of the final lower bound follows from
\begin{align*}
	\nu^2-A_\eta+A_\eta^2
	&=\left(\frac{1-\sqrt{1-4\nu^2}}2-A_\eta\right)
	\left(\frac{1+\sqrt{1-4\nu^2}}2-A_\eta\right)>0.
\end{align*}
Thus for any $t\in[0,1]$, we have
\[
q_n^+(t)q_n^-(t)>0.
\]
Specially, for $t=1$,
\[
\widetilde Q_n= q_n^+(1)q_n^-(1)>0.
\]
Hence we proved if \eqref{eq:stab-smallness} holds, then $\widetilde Q_n>0$.

(ii)\, Since $q_n^\pm(t)=r_{n+1}(t)r_n(t)\pm R_n(t)>0$, we obtain
\[
-1<\frac{R_n(t)}{r_{n+1}(t)r_n(t)}<1.
\]
Define
\begin{align}\label{eq:def}
	\theta_n(t):= \arccos \left( \frac{R_n(t)}{r_{n+1}(t)r_n(t)} \right).
\end{align}
Notice that $r_{n+1}(t)r_n(t)$ and $R_n(t)$ are both quadratic polynomials in $t$, the function $\theta_n(t)$ is differentiable. Differentiating and simplifying yields
\begin{align}
	\theta_n'(t)&=\frac{\cos\theta_n(t)\bigl(r_{n+1}(t)r_n(t)\bigr)'-R_n'(t)}
	{ \sqrt{r_{n+1}^2(t)r_n^2(t)-R_n^2(t)} }\nonumber\\
	&=\frac{\cos\theta_n(t)\bigl(r_{n+1}(t)r_n(t)\bigr)'-R_n'(t)}
	{\sqrt{q_n^+(t)q_n^-(t)}}\label{eq:stab-theta-derivative}.
\end{align}
Expanding the numerator in \eqref{eq:stab-theta-derivative} yields
\[
\bigl(r_n(t)\cos\theta_n(t)-r_{n+1}(t)\bigr)e_{n+1}^0
+\bigl(r_{n+1}(t)\cos\theta_n(t)-r_n(t)\bigr)e_n^0
+d_n(t)e_n^1.
\]
From the definitions of $R_n(t)$ and $\theta_n(t)$, we immediately obtain
\[
d_n^2(t)=r_{n+1}^2(t)+r_n^2(t)-2r_{n+1}(t)r_n(t)\cos \theta_n(t).
\]
A direct computation yields
\[
d_n^2(t)-\big( r_n(t)\cos\theta_n(t)-r_{n+1}(t) \big)^2
=r_n^2(t)\sin^2\theta_n(t)>0.
\]
Therefore, we have
\[
|r_n(t)\cos\theta_n(t)-r_{n+1}(t)|\leq d_n(t),
\]
and similarly,
\[
|r_{n+1}(t)\cos\theta_n(t)-r_n(t)|\leq d_n(t).
\]
Consequently, the numerator in \eqref{eq:stab-theta-derivative} is bounded in absolute value by
\begin{equation}
	d_n(t)\bigl(|e_{n+1}^0|+|e_n^0|+|e_n^1|\bigr).
	\label{eq:stab-theta-numerator}
\end{equation}
Then by the inequality $\sum_{n=0}^{N-1} d_n^2=\sum_{n=0}^{N-1}|h_{n+1}-h_n|^2\leq4E_J$
and Cauchy-Schwarz inequality, we have
\begin{align}
	\sum_{n=0}^{N-1} d_n^2(t)
	&=\sum_{n=0}^{N-1} d_n^2 + 2t\sum_{n=0}^{N-1} d_ne_n^1+t^2\sum_{n=0}^{N-1} |e_n^1|^2\nonumber\\
	&\leq \sum_{n=0}^{N-1} d_n^2+2t\left(\sum_{n=0}^{N-1} d_n^2\right)^{1/2} \left(\sum_{n=0}^{N-1} |e_n^1|^2\right)^{1/2}
	+t^2\sum_{n=0}^{N-1} |e_n^1|^2\nonumber\\
	&\leq(2+t\delta)^2 E_J.\label{eq:stab-d-path}
\end{align}
Moreover, applying the Cauchy-Schwarz inequality once more, we obtain
\begin{align*}
	\left(\sum_{n=0}^{N-1}
	\bigl(|e_{n+1}^0|+|e_n^0|+|e_n^1|\bigr)^2\right)^{1/2}&\leq
	2\left(\sum_{n=0}^N|e_n^0|^2\right)^{1/2}
	+\left(\sum_{n=0}^{N-1}|e_n^1|^2\right)^{1/2}\\
	&\leq\sqrt5\,\delta\sqrt{E_J}.
\end{align*}
Combining this estimate with \eqref{eq:stab-theta-numerator} and
\eqref{eq:stab-d-path}, we have
\begin{align}
	\sum_{n=0}^{N-1}\left|\cos\theta_n(t)\big(r_{n+1}(t)r_n(t)\big)'-R_n'(t)\right|&\le \sum_{n=0}^{N-1}d_n(t)\bigl(|e_{n+1}^0|+|e_n^0|+|e_n^1|\bigr)\nonumber
	\\&\le\sqrt5\,(2+t\delta)\delta E_J\label{eq:stab-total}.
\end{align}
Using \eqref{eq:stab-I-path} and \eqref{eq:stab-total} to \eqref{eq:stab-theta-derivative}, we have
\begin{align*}
	\sum_{n=0}^{N-1}\left|\theta_n'(t)\right|\le \frac{\sqrt5\,(2+t\delta)\delta}{\sqrt{\nu^2-A_\eta+A_\eta^2}}.
\end{align*}
Integrating this inequality over $0\leq t\leq1$ yields
\begin{equation}
	\sum_{n=0}^{N-1}|\theta_n(1)-\theta_n(0)|
	\le
	\frac{\sqrt{5}\,\delta(\frac{\delta}{2}+2)}
	{\sqrt{\nu^2-A_\eta+A_\eta^2}}.
	\label{eq:stab-total-angle}
\end{equation}

Set $\chi_n=\operatorname{sgn}(\operatorname{Im} K_n)$ and
$\gamma_n=\theta_n(1)-\theta_n(0)$. Define
\[
u_n^\sigma=\frac{h_n^\sigma}{r_n},
\qquad
\widetilde u_n=\frac{\widetilde h_n}{\widetilde r_n}.
\]
Using these definitions, we first observe that
\begin{align*}
	h_{n+1}^\sigma \overline{h_n^\sigma}
	=R_n+i\sigma\chi_n S_n
	=r_{n+1}r_n\big( \cos\theta_n(0)+i\sigma\chi_n\sin\theta_n(0) \big)
	=r_{n+1}r_n e^{i\sigma\chi_n\theta_n(0)}.
\end{align*}
From the identity above, we obtain the phase recursion for the exact normalized quantities
\[
u_{n+1}^\sigma=u_{n}^\sigma e^{i\sigma\chi_n\theta_n(0)}.
\]
Similarly, we have
\[
\widetilde u_{n+1}=\widetilde u_{n}
e^{i\sigma\chi_n\theta_n(1)}.
\]
Choose $\lambda_0\in\mathbb T$ such that
$\widetilde u_0=\lambda_0u_0^\sigma$. Iterating the phase recursions then yields
\[
\frac{\widetilde u_n}{\lambda_0u_n^\sigma}
=\exp\left(i\sigma\sum_{j=0}^{n-1}\chi_j\gamma_j\right).
\]
To control the accumulated phase error, define $s_n=\sigma\sum_{j=0}^{n-1} \chi_j\gamma_j$ for $n\ge1$ with $s_0=0$, and let
$$s_{\max}=\max\limits_{0\le n\le N} s_n,\quad s_{\min}=\min\limits_{0\le n\le N} s_n.$$
Let $s_*=(s_{\max}+s_{\min})/2$ and set $\lambda=\lambda_0 e^{is_*}$.
Notice that
\[
|s_n-s_*|\leq \frac{s_{\max}-s_{\min}}{2}\le \frac{\sum_{j=0}^{N-1}|s_{j+1}-s_{j}|}{2}
= \frac{\sum_{j=0}^{N-1}|\gamma_j|}{2}, \quad  0\le n\le N.
\]
Using the elementary inequality $|e^{ix}-e^{iy}|\leq|x-y|$ and combining \eqref{eq:stab-total-angle}, we obtain
\begin{equation}
	|\widetilde u_n-\lambda u_n^\sigma|=|e^{is_n}-e^{is_*}|
	\leq |s_n-s_*| \leq
	\frac{\sqrt5\,\delta(1+\delta/4)}
	{\sqrt{\nu^2-A_\eta+A_\eta^2}},
	\qquad 0\leq n\leq N.
	\label{eq:stab-node-phase}
\end{equation}

For each $n$, we have the exact identity
\[
|\widetilde h_n-\lambda h_n^\sigma|^2
=|\widetilde r_n-r_n|^2
+\widetilde r_nr_n|\widetilde u_n-\lambda u_n^\sigma|^2,
\]
which decomposes the total error into amplitude and phase parts.
By \eqref{eq:stab-noise}, the amplitude part satisfies
\[
\sum_{n=0}^N|\widetilde r_n-r_n|^2\leq\delta^2E_J.
\]
For the phase part, we first bound the weight factor,
\begin{align*}
	\sum_{n=0}^N\widetilde r_nr_n&=\sum_{n=0}^N(r_n+e_n^0)r_n
	=E_J+\sum_{n=0}^Ne_n^0r_n\\
	&\leq E_J+
	\left(\sum_{n=0}^N|e_n^0|^2\right)^{1/2}
	\left(\sum_{n=0}^Nr_n^2\right)^{1/2}\\
	&\leq(1+\delta)E_J.
\end{align*}
Combining this with \eqref{eq:stab-node-phase}, we obtain
\begin{align*}
	\sum_{n=0}^N\widetilde r_nr_n|\widetilde u_n-\lambda u_n^\sigma|^2\leq(1+\delta)E_J\frac{5\delta^2(1+\delta/4)^2}
	{\nu^2-A_\eta+A_\eta^2}.
\end{align*}
Adding the amplitude and phase contributions yields
\begin{align*}
	\sum_{n=0}^{N}
	|\widetilde h_n-\lambda h_n^\sigma|^2
	\leq
	\delta^2\left(
	1+\frac{5(1+\delta)(1+\delta/4)^2}
	{\nu^2-A_\eta+A_\eta^2}
	\right)E_J,
\end{align*}
which proves \eqref{eq:stab-core-bound}.

(iii)\, It remains to verify the sufficient branch condition.  Set
$\varepsilon_\eta=2\sqrt2\,\eta+\eta^2$ as in the theorem. First note that
$r_n \le \sqrt{E_J}$ for every $n\in J$ from the definition of $E_J$.
Also, since $d_n=|h_{n+1}-h_n|\le r_{n+1}+r_n$, the Cauchy-Schwartz inequality gives
\[
d_n^2 \le (r_{n+1}+r_n)^2 \le 2(r_{n+1}^2+r_n^2) \le 2E_J,
\]
so $d_n \le \sqrt{2E_J}$. The same argument yields $\ell_n \le \sqrt{2E_J}$.
Using these bounds together with \eqref{eq:stab-noise-max}, we now control the squared perturbations,
\[
|\widetilde r_n^2-r_n^2| =|2r_ne_n^0+(e_n^0)^2| \le 2|r_n||e_n^0|+|e_n^0|^2\le 2\sqrt{E_J}\cdot\eta\sqrt{E_J} + \eta^2E_J = (2\eta+\eta^2)E_J\le \varepsilon_\eta E_J.
\]
Similarly, we have
\[
|\widetilde d_n^2-d_n^2|,\quad |\widetilde \ell_n^2-\ell_n^2|
\le\varepsilon_\eta E_J.
\]
Consequently, from the preceding squared perturbation bounds we obtain
\[
|\widetilde R_n-R_n|,
\quad |\widetilde C_n-C_n|
\leq\frac32\varepsilon_\eta E_J.
\]
Meanwhile, we note that
\[
|R_n|\le \frac{r_{n+1}^2+r_n^2}{2}\le \frac{E_J}{2},
\qquad
|C_n|\le \frac{r_{n+2}^2+r_n^2}{2}\le \frac{E_J}{2},
\]
then we have
\begin{align*}
	&\quad|\widetilde R_{n+1}\widetilde R_n-R_{n+1}R_n|\\&= |(\widetilde R_{n+1}-R_{n+1})(\widetilde R_n-R_n) + R_{n+1}(\widetilde R_n-R_n)
	+ R_n(\widetilde R_{n+1}-R_{n+1})|\\
	&\leq
	|\widetilde R_{n+1}-R_{n+1}|
	|\widetilde R_n-R_n|+|R_{n+1}||\widetilde R_n-R_n|
	+|R_n||\widetilde R_{n+1}-R_{n+1}|
	\\
	&\leq
	\left(\frac32\varepsilon_\eta
	+\frac94\varepsilon_\eta^2\right)E_J^2.
\end{align*}
Similarly,
\begin{align*}
	|\widetilde r_{n+1}^{2}\widetilde C_n-r_{n+1}^2C_n|\leq
	\left(2\varepsilon_\eta
	+\frac32\varepsilon_\eta^2\right)E_J^2.
\end{align*}
Then combining the definition of $G_n$, these two estimates give
\begin{equation}
	|\widetilde G_n-G_n|\le|\widetilde R_{n+1}\widetilde R_n-R_{n+1}R_n|+|\widetilde r_{n+1}^{2}\widetilde C_n-r_{n+1}^2C_n|
	\leq\left(\frac72\varepsilon_\eta
	+\frac{15}{4}\varepsilon_\eta^2\right)E_J^2.
	\label{eq:stab-G-error}
\end{equation}
On the other hand, \eqref{eq:stab-G} and
\eqref{eq:stab-transversality} give
\[
|G_n|=|\operatorname{Im} K_{n+1}\operatorname{Im} K_n|=S_{n+1}S_n\geq\nu^2E_J^2.
\]
Thus \eqref{eq:stab-zero-branch-condition} and \eqref{eq:stab-G-error}
imply
\[
|\widetilde G_n-G_n|< |G_n|.
\]
Consequently,
\[
\operatorname{sgn}(\widetilde G_n)=\operatorname{sgn}(G_n)
=\operatorname{sgn}(\operatorname{Im} K_{n+1}\operatorname{Im} K_n),
\qquad 0\leq n\leq N-2.
\]
Moreover, since
\[
\sqrt{1+\sqrt{1-4\nu^2}}\leq\sqrt2,
\]
we have
\[
A_\eta\leq3\sqrt2\,\eta+2\eta^2
\leq\frac72\varepsilon_\eta.
\]
Also, the inequality
\[
\nu^2\leq\frac{1-\sqrt{1-4\nu^2}}2
\]
holds for all $0<\nu\le1/2$. Hence \eqref{eq:stab-zero-branch-condition} implies
\eqref{eq:stab-smallness}. In particular, we have
$\widetilde S_{n+1}\widetilde S_n>0$. Therefore, in Step~7 of Algorithm~\ref{alo7} the minimizer is determined solely by the sign of $\widetilde G_n$.

Since we have already established $\operatorname{sgn}(\widetilde G_n)=\operatorname{sgn}(\operatorname{Im} K_{n+1}\operatorname{Im} K_n)$, the recursion propagates the sign correctly: starting from either choice of  $\widetilde\varepsilon_0$, we obtain either $\widetilde\varepsilon_n=\operatorname{sgn}(\operatorname{Im} K_n)$ for all  $n$, or its global negative. Thus the branch condition \eqref{eq:stab-branch} holds.
\end{proof}
	
	\begin{theorem}[Stability on a fixed time interval]
		\label{thm:stab-full-gm}
		Let $0<b\leq1/(2B)$, and let $J$, $E_{\rm all}$, $E_J$ and
		$\alpha_J$ be defined as in \eqref{eq:Eall}--\eqref{eq:AJ}. Suppose that the noise conditions
		\eqref{eq:stab-noise} and \eqref{eq:stab-noise-max} hold on $J$,
		and that the transversality condition \eqref{eq:stab-transversality}, the branch condition
		\eqref{eq:stab-branch}, and the smallness condition
		\eqref{eq:stab-smallness} also hold. Outside $J$, impose no condition on the
		reconstructed phases: let $\widetilde h_n$ be arbitrary subject only to
		\[
		|\widetilde h_n|=\widetilde r_n,
		\qquad n\notin J.
		\]
		Additionally, assume that the magnitude noise outside $J$ satisfies
		\begin{equation}
			\sum_{n\notin J}|e_n^0|^2
			\leq\delta_{\rm out}^2E_{\rm all}.
			\label{eq:stab-outside-noise}
		\end{equation}
		Define $I_b=\left[-1/(2b),1/(2b)\right]$ and for $\sigma=\pm1$, set
		\[
		g^\sigma(t)=
		\begin{cases}
			g(t),&\sigma=1,\\
			\overline{g(-t)},&\sigma=-1.
		\end{cases}
		\]
		Also define $\widetilde g$ by its Fourier expansion on $I_b$,
		\[
		\widetilde g(t)=
		\begin{cases}
			b\displaystyle\sum_{n\in\mathbb Z}
			\widetilde h_ne^{2\pi i(\beta+nb)t},&t\in I_b,\\[1ex]
			0,&t\notin I_b,
		\end{cases}
		\]
		with convergence in $L^2(I_b)$.
		Then the following error bound holds:
		\begin{align}
			\frac{
				\displaystyle\min_{\substack{\lambda\in\mathbb T\\
						\sigma\in\{-1,1\}}}
				\|\widetilde g-\lambda g^\sigma\|_{L^2(\mathbb R)}^2}
			{\|g\|_{L^2(\mathbb R)}^2}\quad\leq
			\alpha_J\delta^2\left(
			1+\frac{5(1+\delta)(1+\delta/4)^2}
			{\nu^2-A_\eta+A_\eta^2}
			\right)
			+\left(2\sqrt{1-\alpha_J}+\delta_{\rm out}\right)^2.
			\label{eq:stab-gm-bound}
		\end{align}
	\end{theorem}

	\begin{proof}
		By Theorem~\ref{thm:stab-core}, there exist some
		$\lambda\in\mathbb T$ and $\sigma\in\{-1,1\}$ such that
		\begin{align}
			\sum_{n\in J}|\widetilde h_n-\lambda h_n^\sigma|^2
			&\leq
			\delta^2E_J\left(
			1+\frac{5(1+\delta)(1+\delta/4)^2}
			{\nu^2-A_\eta+A_\eta^2}
			\right)\nonumber\\
			&=\alpha_J\delta^2\left(
			1+\frac{5(1+\delta)(1+\delta/4)^2}
			{\nu^2-A_\eta+A_\eta^2}
			\right)E_{\rm all}.
			\label{eq:stab-gm-core-part}
		\end{align}
		For the same $\lambda$ and $\sigma$, no phase information is used outside
		$J$. Since $|h_n^\sigma|=r_n$ and
		$|\widetilde h_n|=\widetilde r_n$, applying Cauchy-Schwarz inequality we have
		\begin{align}
			\left(\sum_{n\notin J}
			|\widetilde h_n-\lambda h_n^\sigma|^2\right)^{1/2}
			&\leq
			\left(\sum_{n\notin J}\widetilde r_n^2\right)^{1/2}
			+\left(\sum_{n\notin J}r_n^2\right)^{1/2}\nonumber\\
			&\leq
			2\left(\sum_{n\notin J}r_n^2\right)^{1/2}+\left(\sum_{n\notin J}|e_n^0|^2\right)^{1/2}\nonumber\\
			&\leq
			2\left(E_{\rm all}-E_J\right)^{1/2}
			+\left(\sum_{n\notin J}|e_n^0|^2\right)^{1/2}\nonumber\\
			&\leq
			\left(2\sqrt{1-\alpha_J}+\delta_{\rm out}\right)
			\sqrt{E_{\rm all}}. \label{eq:stab-gm-outside-part}
		\end{align}
		Since $\operatorname{supp}g\subset[-B,B]\subset I_b$, the function
		$g^\sigma$ can be represented on $I_b$ by its Fourier series
		\[
		g^\sigma(t)=b\sum_{n\in\mathbb Z}
		h_n^\sigma e^{2\pi i(\beta+nb)t},
		\qquad t\in I_b.
		\]
		The system
		$\{\sqrt b\,e^{2\pi i(\beta+nb)t}\}_{n\in\mathbb Z}$ forms an
		orthonormal basis of $L^2(I_b)$. Therefore, by Parseval's identity,
		\begin{equation}
			\|\widetilde g-\lambda g^\sigma\|_{L^2(\mathbb R)}^2
			=b\sum_{n\in\mathbb Z}
			|\widetilde h_n-\lambda h_n^\sigma|^2,
			\qquad
			\|g\|_{L^2(\mathbb R)}^2=bE_{\rm all}.
			\label{eq:stab-gm-parseval}
		\end{equation}
		Substituting the bounds from \eqref{eq:stab-gm-core-part} and \eqref{eq:stab-gm-outside-part} into \eqref{eq:stab-gm-parseval} yields \eqref{eq:stab-gm-bound}.
	\end{proof}

	\section{Numerical Experiments}
	In this section, we design some numerical experiments to validate the three retrieval algorithms described in Sections~3--4. Let's first introduce some geometric and numerical parameters.
	
	For a fixed time shift $m$, we define the windowed local function $g_m$ and its Fourier transform as \eqref{eq:wd}. The marks in this section are consistent with those in section 4.
	
	\subsection{Implementation of Local Recovery Algorithms}
	For different algorithms with fixed m, we introduce their numerical simulation process respectively.
	
	\subsubsection*{Two-window method.}
	This method begins with truncated Shannon interpolation of the two square magnitudes $$\{|\mathcal{V}_{\phi} f(ma,n\mathfrak{b})|^2,|\mathcal{V}_{\psi} f(ma,n\mathfrak{b})|^2:n\in[N_1,N_2]\},$$ where $\psi(\cdot)=\phi(x)(e^{2 \pi ib\cdot }-1) $. Then at points $x_n=\beta+nb$ with $\beta$ chosen from a discrete set of candidates (e.g., equally spaced in $[0,b)$), we obtain $r_n$ and $d_n$. Then we compute $R_n=\operatorname{Re} K_n$ via \eqref{equ3-1}.
	
	To select the optimal offset $\beta$ for each time slice, we define a score
	\[
	S_{\text{2W}}(\beta) = \frac{\min_n r_n^2}{\operatorname{median}_n r_n^2+ \varsigma}
	- 2 \cdot \frac{\|\max\bigl( |R_n| - |K_n|,\; 0 \bigr)\|_2}{\max\bigl( \|K_n\|_2,\; \varsigma \bigr)}
	\]
	where  $\varsigma>0$ is a small constant to avoid division by zero and $\|\cdot\|_2$ denotes the $\ell_2$-norm. The first term discourages lattices with atypically small squared amplitudes, while the second penalizes breaches of $|R_n| \le |K_n|$. The offset $\beta$ that maximises $S_{\text{2W}}(\beta)$ is selected independently per-slice.
	
	Once the optimal offset $\beta$ is fixed, we compute $Q_n=(\operatorname{Im} K_n)^2$ via \eqref{equ3-2}, where $R_n$ is projected onto $[-|K_n|,|K_n|]=[-|r_{n+1}r_n|,|r_{n+1}r_n|]$ to enforce $Q_n\ge0$. To determine the signs of the square roots of $ Q_n $,
	we introduce a refined grid over the interval spanned by the selected lattice points. More precisely, let  $x_0=x_{\min}$ and  $x_N=x_{\max}$ be the endpoints of the selected lattice segment, and define the refined grid
	\[
	y_j=x_0+\frac{b}{N_{\mathrm{ref}}} j,\qquad j=0,1,\dots,NN_{\mathrm{ref}},
	\]
	where  $N_{\mathrm{ref}}\in\mathbb{N}$ is a prescribed refinement factor.
	
	On this fine grid we interpolate the squared magnitudes and compute $Q_j $, then apply a spatial numerical continuation strategy to select the signed square root $I_j=\pm\sqrt{Q_j} $. The procedure is initialised at the node where $Q_j$ attains its global maximum; at this anchor point and at both of its immediate neighbours, we assign the positive branch, i.e.  $I=+\sqrt{Q}$. For rightward propagation, a second-order linear predictor is constructed from the two previously established solutions:
	\[
	I_k^{\mathrm{pred}} = 2I_{k-1} - I_{k-2}.
	\]
	The signed solution at the current node $k$ is then selected from the two candidates $\pm\sqrt{Q_k}$ by minimising the absolute deviation from this prediction:
	\[
	I_k = \arg \min_{s\in\{\pm\sqrt{Q_k}\}} \bigl|s - I_k^{\mathrm{pred}}\bigr|.
	\]
	Leftward propagation is handled symmetrically using the corresponding backward extrapolation.
	
	Since  $x_n=y_{nN_{\mathrm{ref}}} $, the selected lattice is contained in the fine grid. Once the full signed sequence  $I_j$ has been obtained, we retain its branch signs on the selected lattice and set
	\[
	\varepsilon_n
	=
	\operatorname{sgn}\bigl(I_{nN_{\mathrm{ref}}}\bigr),
	\qquad
	I_n=\varepsilon_n\sqrt{Q_n},
	\]
	where  $\operatorname{sgn}(0)=0 $. The complex spectral values  $h_n$ are then recovered via $K_n=R_n+iI_n$ and the recursion \eqref{equ3-3}--\eqref{equ3-4}, starting from the node with largest $r_n $.
	
	Since $R_n$ is projected onto $[-|K_n|,|K_n|]$ before $Q_n$ is formed, the identity  $|\widetilde K_n|=r_{n+1}r_n$ holds by construction. Nevertheless, to avoid floating-point magnitude drift along the chain and ill-conditioning when $|\widetilde h_{n-1}|$ is small, we introduce the following magnitude-preserving form of the bidirectional recursion: $\widetilde K_n$ supplies only the phase increment, while the measured magnitude $r_n$ fixes the recovered magnitude,
	\begin{equation}\label{eq:magnitude-preserving-forward}
		\widetilde h_n=r_n\exp\left(
		i\arg\frac{\widetilde K_{n-1}}
		{\overline{\widetilde h_{n-1}}}
		\right),
	\end{equation}
	with the backward update performed analogously. This coincides with \eqref{equ3-3}--\eqref{equ3-4} in exact arithmetic.

	\subsubsection*{Three-window $K$-sign method.}
	This method begins with truncated Shannon interpolation of the three squared magnitudes
	\[
	\{|\mathcal{V}_{\phi} f(ma,n\mathfrak{b})|^2,\ |\mathcal{V}_{\psi_1} f(ma,n\mathfrak{b})|^2,\ |\mathcal{V}_{\psi_2} f(ma,n\mathfrak{b})|^2 : n\in[N_1,N_2]\},
	\]
	where $\psi_k(\cdot)=\phi(x)(e^{2 \pi ikb\cdot }-1)$. Interpolation to the same candidate lattices $x_n=\beta+nb$ yields $r_n, d_n, \ell_n$
	Let
	\[
	\zeta = \max\{\max_n r_n,\ \max_n d_n,\ \max_n \ell_n,\ \varsigma\}
	\]
	be a global scale factor, where  $\varsigma>0$ is a small constant to ensure a positive lower bound.
	
	Let $\theta_n$ denote the phase difference between adjacent points, defined as
	\[
	\theta_n=\arccos\!\left(\frac{r_{n+1}^2+r_n^2-d_n^2}{2r_{n+1}r_n}\right),
	\]
	which is the noiseless value of  $\theta_n(0)$ in \eqref{eq:def}.
	
	For each candidate $\beta $, we define a pre-reconstruction score that measures how well the measured third chain $\ell$ can be predicted from the first two chains under the two possible sign transitions. For each triple  $(r_{n-1}, r_n, r_{n+1})$, the two predictions of  $\ell_{n+1}$ corresponding to same-sign ($\rho=+1 $) and opposite-sign ($\rho=-1 $) transitions are
	\[
	\ell_{n+1}^{\mathrm{same}}
	=\sqrt{r_{n-1}^2 + r_{n+1}^2 - 2r_{n-1}r_{n+1}\cos(\theta_{n-1}+\theta_{n})},
	\]
	\[
	\ell_{n+1}^{\mathrm{opp}}
	=\sqrt{r_{n-1}^2 + r_{n+1}^2 - 2r_{n-1}r_{n+1}\cos(\theta_{n-1}-\theta_{n})}.
	\]
	The score for the candidate $\beta$ is defined as
	\[
	S_{\mathrm{KS}}(\beta)=\frac{1}{M}\sum_{n=1}^{M} w_n\cdot
	\left|
	\frac{|\ell_{n+1}-\ell_{n+1}^{\mathrm{same}}|}
	{\ell_{n+1}+\ell_{n+1}^{\mathrm{same}}+\varrho}
	-
	\frac{|\ell_{n+1}-\ell_{n+1}^{\mathrm{opp}}|}
	{\ell_{n+1}+\ell_{n+1}^{\mathrm{opp}}+\varrho}
	\right|,
	\]
	where $M$ is the number of valid triples, $\varrho=10^{-12}\zeta$ and
	\[
	w_n=\min\!\left(1,\frac{r_{n-1}^2+r_n^2+r_{n+1}^2}{3\zeta^2}\right)
	\frac{|\ell_{n+1}^{\text{same}} - \ell_{n+1}^{\text{opp}}|}
	{\ell_{n+1}^{\text{same}} + \ell_{n+1}^{\text{opp}} + \varrho}.
	\]
	The offset $\beta$ that maximises $S_{\mathrm{KS}}(\beta)$ is selected independently per-slice.
	
	On the selected lattice, we compute $R_n$, $Q_n$ and $C_n$ and set
	$S_n = \sqrt{\max (Q_n,\;0)}$. The relative sign $\rho_n=\varepsilon_n\varepsilon_{n+1}$ is determined pointwise by comparing
	\begin{align*}
		E_n^{\mathrm{same}}=\left|G_n-S_n S_{n+1}\right|,\quad
		E_n^{\mathrm{opp}}=\left|G_n+S_n S_{n+1}\right|,
	\end{align*}
	We set $\rho_n=+1$ if $E_n^{\mathrm{same}}\le E_n^{\mathrm{opp}} $, otherwise $\rho_n=-1 $. The signs $\varepsilon_n$ are propagated from the edge where $S_n$ is maximal, anchored at $\varepsilon_q=+1 $; if all $S_n=0 $, we set all $\varepsilon_n=0 $. This yields the full sequence
	\[
	\widetilde K_n=R_n+i\,\varepsilon_n S_n.
	\]
	Since $S_n$ is obtained by clipping $Q_n$ to enforce $Q_n\ge0 $, the identity $|\widetilde K_n|=r_n r_{n+1}$ may fail. We therefore apply the same magnitude-preserving bidirectional phase recursion \eqref{eq:magnitude-preserving-forward} introduced for the two-window method, with $\widetilde K_n$ supplying only the phase increment and $r_n$ fixing the recovered magnitude.
	
	\subsubsection*{Three-window SCR method.}
	This method also begins with truncated Shannon interpolation of the same three squared magnitudes as in the $K$-sign method, evaluated on each candidate lattice $x_n=\beta+nb$, to obtain the chains $r_n, d_n, \ell_n $. For each candidate $\beta $, we form the six-dimensional magnitude vector
	\[
	\mathbf{y}_n = \bigl(r_n,\; r_{n+1},\; r_{n+2},\; d_n,\; \ell_{n},\; d_{n+1}\bigr)^{\mathsf{T}},
	\]
	which is the SCR column of Algorithm~5. Then we compute $R_n$, $Q_n$, $C_n$ and
	\[
	\Delta_n=r_{n+2}^2 r_{n}^2-C_{n}^2.
	\]
	
	To select the optimal offset $\beta$, we define
	\[
	S_{\mathrm{SCR}}(\beta)=\frac{1}{M}\sum_{n=1}^{M}
	\frac{\sqrt{\max(Q_{n+1},0)+\max(\Delta_{n},0)+\max(Q_{n},0)}}
	{1+\nu_n},
	\]
	where $M$ is the number of valid triples and \[
	\nu_n=
	\frac{\sqrt{\max(-Q_{n+1},0)+\max(-\Delta_n,0)+\max(-Q_n,0)}}
	{r_{n+1}^2+r_n^2+r_{n+2}^2+\varsigma}.
	\]
	The offset $\beta$ that maximises $S_{\mathrm{SCR}}(\beta)$ is selected independently per-slice.
	
	We then apply the Gerchberg-Saxton alternating projection algorithm independently to each $\mathbf{y}_n$ to obtain complex estimates of the triple $(h_n, h_{n+1}, h_{n+2})$. These estimates are inherently ambiguous up to an unknown global phase and conjugation. We exploit that adjacent columns $n$ and $n+1$ share two common points to resolve these ambiguities.
	
	\subsection{Global Assembly via Adjacent Stitching}
	
	After the local retrieval on each time slice $m$, we obtain a complex estimate $\{\tilde{h}_{m,n}:n=N_1',\dots,N_2'\}$ which is ambiguous up to a global phase and conjugation. To construct the two possible candidates, we use the finite Fourier series
	\[
	\tilde g_m^{(1)}(u)=b\sum_n \tilde h_{m,n} e^{2\pi i x_{m,n}u},\qquad
	\tilde g_m^{(-1)}(u)=b\sum_n \overline{\tilde {h}_{m,n}} e^{2\pi i x_{m,n}u}=\overline{\tilde g_m^{(1)}(-u)}.
	\]
	In practice, the series is evaluated in blocks of $256$ local time points to avoid generating excessively large exponential matrices.
	
	\subsubsection*{Pairwise alignment between adjacent slices.}
	To align a current slice to an already-fixed reference slice, we first recover the underlying function $f$ in their overlap region by removing the window. Let $c_r$ and $c_m$ be the centres of the reference and current slices, respectively. On the geometric overlap $\Omega_{r,m}$, we compute
	\[
	F_r(t)=\frac{\tilde g_r(t-c_r)}{\overline{\phi(t-c_r)}},\qquad
	F_{m,s}(t)=\frac{\tilde g_m^{(\sigma)}(t-c_m)}{\overline{\phi(t-c_m)}},
	\]
	retaining only points where both window values exceed a small threshold relative to $\max|\phi|$.
	
	For each candidate $\sigma\in\{-1,1\}$ of the current slice, the cross-correlation sum is
	\[
	z_{m,\sigma}=\sum_{t\in\Omega_{r,m}} F_r(t)\overline{F_{m,\sigma}(t)}.
	\]
	The optimal unit phase aligning the candidate to the reference is
	\[
	\lambda_{m,\sigma}=
	\begin{cases}
		z_{m,\sigma}/|z_{m,\sigma}|,& z_{m,\sigma}\neq 0,\\
		1,& z_{m,\sigma}=0,
	\end{cases}
	\]
	and the normalised residual is defined as
	\[
	\rho_{m,\sigma}=
	\frac{\sum_{\Omega_{r,m}} |F_r-\lambda_{m,\sigma}F_{m,\sigma}|^2}
	{\sum_{\Omega_{r,m}} |F_r|^2+\sum_{\Omega_{r,m}} |F_{m,\sigma}|^2}.
	\]
	The candidate with the smaller residual is selected:
	\[
	\sigma_m=\arg\min_{\sigma\in\{-1,1\}}\rho_{m,\sigma},\qquad \tilde g_m = \lambda_{m,\sigma_m}\tilde g_m^{(\sigma_m)}.
	\]
	
	\subsubsection*{Anchor selection and sequential propagation.}
	The anchor slice is chosen as the one with smallest $|m|$. Its orientation is determined by comparing both candidates against its immediate neighbours using the pairwise rule above. For each anchor candidate $\sigma$, we compute the minimal residual achievable by the left and right neighbours (each over its two candidates) and sum them:
	\[
	C_{\text{anchor}}(\sigma)=\rho_{\text{left}}^{\min}(\sigma)+\rho_{\text{right}}^{\min}(\sigma).
	\]
	The anchor direction is then selected as
	\[
	\sigma_{\text{anchor}}=\arg\min_{\sigma\in\{-1,1\}} C_{\text{anchor}}(\sigma).
	\]
	If only one slice exists, the first candidate is taken by default.
	
	Once the anchor is fixed, propagation proceeds sequentially to the right and then to the left, applying the pairwise alignment rule at each step using only the already-fixed direct neighbour.
	
	\subsubsection*{Frame weights and final reconstruction.}
	The selected residual $\rho_m$ defines the frame weight
	\[
	w_m=\frac{1}{1+\rho_m},\qquad w_{\text{anchor}}=1.
	\]
	The final reconstruction is obtained by pointwise weighted overlap averaging:
	\begin{align}\label{equ5-1}
		\tilde f(t)= \frac{\sum_m w_m\,\phi(t-ma)\,\tilde g_m(t-ma)}
		{\sum_m w_m |\phi(t-ma)|^2}.
	\end{align}
	\subsection{Synthetic examples.}
	To demonstrate the performance of the three methods, we conduct specific experiments with the rectangular window $\phi=\chi_{[-B,B]}$.
	
	The synthetic signals are compactly supported on a finite interval $[A,C]$ and consists of two segments joined at a discontinuity:
	\begin{align}\label{eq:f}
		f_{\mathrm{raw}}(t)=
		\begin{cases}
			r(t), & A<t<\tau,\\[4pt]
			q(t), & \tau\le t<C,
		\end{cases}
	\end{align}
	where the left segment is an exponentially modulated complex exponential
	\[
	r(t)=\gamma e^{s\lambda(t-m_r)}e^{2\pi i\,\xi_0 t},
	\]
	and the right segment is a nonharmonic Fourier sum
	\[
	q(t)=\sum_{k=1}^{K} a_ke^{2\pi i\,\xi_k(t-\tau)}.
	\]
	The parameters $\tau,s,\lambda,m_r,\xi_0,\gamma,K,\xi_k,a_k$ are randomly generated subject to the constraints described in Section~3. The signal is normalised to have unit $\mathrm{RMS}$ over its support:
	\[
	f(t)=\frac{f_{\mathrm{raw}}(t)}{\|f_{\mathrm{raw}}\|_{\mathrm{RMS}}},\qquad
	\|f_{\mathrm{raw}}\|_{\mathrm{RMS}}=\sqrt{\frac{1}{C-A}\int_A^C |f_{\mathrm{raw}}(t)|^2\,dt}.
	\]
	This normalisation ensures that the average power of the test signal is always unity, providing a consistent basis for comparing errors across different random realisations.
	
	Additive absolute-uniform noise is applied to the measured magnitudes:
	\[
	\tilde r_n = r_n + e_n^0,\qquad
	\tilde d_n = d_n + e_n^1,\qquad
	\tilde \ell_n = \ell_n + e_n^2,
	\]
	where $e_n^0, e_n^1, e_n^2$ denote independent noise realisations with  $e_n^k\sim \mathcal U[-\epsilon,\epsilon]$, and all noisy magnitudes are clipped to nonnegative values. This noise model does not depend on the signal amplitude; the RMS normalisation nonetheless makes the relative error meaningful by keeping the signal magnitude scale consistent across trials.
	
	The frame parameters are $B=1$, $a=0.25$, $[N_1,N_2]=[-20/\mathfrak{b},20/\mathfrak{b}]$  with signal support $[A,C]=[-2,3]$. The slice indices are chosen as $m\in[-11,15]$ so that the collection of windows $\{[ma-B, ma+B]\}_{m\in[-11,15]}$ covers the full support $[-2,3]$ with redundancy.
	
	The original function used in the numerical simulation is the function  $f_{\mathrm{raw}}$ defined as \eqref{eq:f}, i.e., with the following parameters:
	\[
	\begin{aligned}
		&\|f_{\mathrm{raw}}\|_{\mathrm{RMS}}=1.418892266800545,\quad \tau=0.65446554909395371,\\
		&\gamma=1.018986660670909-0.13948982640056548\,i,\quad s=1,\\
		&\lambda=0.99843079625278031,\quad m_r=-0.67276722545302314,\\
		&\xi_0=-2.5086703942485098,\quad K=4,
	\end{aligned}
	\]
	and
	\[
	\begin{aligned}
		a_1&=-0.6312650835795649-0.024533957769310277\,i,\\
		a_2&=-0.15534360399466846+0.32076248832155635\,i,\\
		a_3&=0.04410918701954139-0.43546887883371682\,i,\\
		a_4&=-0.03102852081826258+0.62075182570221987\,i,
	\end{aligned}
	\]
	\[
	\begin{aligned}
		\xi_1&=-4.3491371598247008,\quad
		\xi_2=-1.8183545993362658,\\
		\xi_3&=-0.9473855099026447,\quad
		\xi_4=0.91690508947210159.
	\end{aligned}
	\]

	The reconstruction error $\mathcal{E}$ is defined as the relative $L^2$ error between the reconstructed signal $\tilde f$ and the synthetic signal $f$ up to a global phase:
	\[
	\mathcal{E} = \frac{\min_{|\lambda|=1} \| \lambda \tilde f - f \|_{L^2([-2,3])}}{\|f\|_{L^2([-2,3])}}.
	\]
	
	\subsubsection*{Two-window result}
	
	The test uses the rectangular window
	\[
	\phi(u)=\mathbf{1}_{[-B,B]}(u),\qquad B=1.
	\]
	The second window is
	\[
	\psi(u)=\phi(u)\bigl(e^{2\pi \mathrm{i} b u}-1\bigr),\qquad b=0.2.
	\]
	The candidate offsets $\beta\in\{0,b/8,2b/8,\dots,7b/8\}$ and the refined grid parameter is set to  $N_{\mathrm{ref}}=8 $.
	
	Figures~1--3 show representative reconstructions obtained by Algorithm~\ref{alo2} with $\mathfrak{b}=0.2$ and absolute-uniform noise at levels $\epsilon=0$, $0.001$, and $0.01$ respectively.
	\input{two-window-uniform_eps-0.tex}
	\input{two-window-uniform_eps-0.001.tex}
	\input{two-window-uniform_eps-0.01.tex}
	\subsubsection*{Three-window K-sign result}
	Based on the two-window setting, a third window is added:
	\[
	\psi_2(u)=\phi(u)\bigl(e^{4\pi\mathrm{i}bu}-1\bigr),\quad b=0.2.
	\]
	The candidate offsets $\beta\in\{0,b/8,2b/8,\dots,7b/8\}$.
	
	Figures~4--6 show representative reconstructions obtained by Algorithm~\ref{alo8} with $\mathfrak{b}=0.2$ and absolute-uniform noise at levels $\epsilon=0$, $0.001$, and $0.01$ respectively.
	\input{new-k-sign-uniform_eps-0.tex}
	\input{new-k-sign-uniform_eps-0.001.tex}
	\input{new-k-sign-uniform_eps-0.01.tex}
	\subsubsection*{Three-window SCR result}
	The same third window and candidate offsets as above are used. Figures~7--9 show representative reconstructions obtained by Algorithm~\ref{alo4} with $\mathfrak{b}=0.2$ and absolute-uniform noise at levels $\epsilon=0$, $0.001$, and $0.01$ respectively.
	\input{old-six-scr-uniform_eps-0.tex}
	\input{old-six-scr-uniform_eps-0.001.tex}
	\input{old-six-scr-uniform_eps-0.01.tex}
	
	\subsubsection*{Runtime comparison}
	All numerical experiments were performed under the same hardware and software environment. To compare the computational efficiency of the three-window methods, we measured the runtime of the K-sign method and the SCR method under identical conditions. For each noise level $\epsilon \in \{0,0.001,0.01\}$, we generated $4$ independent noise realizations. Under each realization, both algorithms were run $5$ times independently, and the runtime was recorded and averaged.
	\begin{table}[htbp]
		\centering
		\caption{Runtime comparison: K-sign vs. SCR.}
		\label{tab:runtime}
		\begin{tabular}{cccc}
			\toprule
			Noise level $\varepsilon$ & K-sign time (s) & SCR time (s) & Speedup (SCR / K-sign) \\
			\midrule
			$0$      & $0.71$ & $100.58$ & $141.66\times$ \\
			$0.001$  & $0.78$ & $227.15$ & $291.22\times$ \\
			$0.01$   & $0.70$ & $889.19$ & $1270.27\times$  \\
			\bottomrule
		\end{tabular}
	\end{table}
	
	The numerical results show that the three methods exhibit different performance characteristics. The two-window method is more sensitive to noise, since an incorrect continuation of the square-root branch near a zero of $Q$ may affect a substantial portion of the recovered sequence. The K-sign method attains comparable accuracy with a substantially shorter runtime. However, since the signs are propagated recursively, an incorrect relative-sign decision may reverse the subsequent sign sequence on the same side of the anchor and lead to a large accumulated phase error. The SCR method shows a slight overall advantage in reconstruction accuracy in the present experiments, but it requires considerably more runtime and relies on the Gerchberg-Saxton iteration. As explicitly observed by Lai, Littmann, and Weber \cite{2021Conjugate}, this iteration may fail to reach the desired solution for some initialisations and may exhibit the traps and tunnels phenomenon.
	
%\bibliography{Reference1008}

\begin{thebibliography}{10}

\bibitem{2021stft}
R.~Alaifari and M.~Wellershoff.
\newblock Uniqueness of {STFT} phase retrieval for bandlimited functions.
\newblock {\em Appl. Comput. Harmon. Anal.}, 50:34--48, 2021.

\bibitem{2022gabor}
R.~Alaifari and M.~Wellershoff.
\newblock Phase retrieval from sampled {G}abor transform magnitudes:
  counterexamples.
\newblock {\em J. Fourier Anal. Appl.}, 28(1):Paper No. 9, 8, 2022.

\bibitem{Alaifari2024MultiWindow}
R.~Alaifari and Y.~Yang.
\newblock Multi-window approaches for direct and stable stft phase retrieval,
  2024.
\newblock arXiv:2410.05486 [math.FA].

\bibitem{Entire2014}
N.~B. Andersen.
\newblock Entire {$L^p$}-functions of exponential type.
\newblock {\em Expo. Math.}, 32(3):199--220, 2014.

\bibitem{Bartusel2023}
D.~Bartusel.
\newblock Injectivity conditions for {STFT} phase retrieval on {$\mathbb Z$},
  {$\mathbb Z_d$} and {$\mathbb R^d$}.
\newblock {\em J. Fourier Anal. Appl.}, 29(4):Paper No. 53, 35, 2023.

\bibitem{2018stft}
T.~Bendory, Y.~C. Eldar, and N.~Boumal.
\newblock Non-convex phase retrieval from {STFT} measurements.
\newblock {\em IEEE Trans. Inform. Theory}, 64(1):467--484, 2018.

\bibitem{Chen2026STFTACHA}
T.~Chen, H.~Lu, W.~Sun, and Y.~Zhao.
\newblock {STFT} phase retrieval with two window functions.
\newblock {\em Applied and Computational Harmonic Analysis}, 85:101905, 2026.

\bibitem{Cheng2025Conjugate}
C.~Cheng, B.~Wu, and J.~Xian.
\newblock Conjugate phase retrieval on graphs and with applications in
  shift-invariant spaces, 2025.
\newblock arXiv:2507.22468 [math.FA].

\bibitem{1987ast}
J.~C. Dainty and J.~R. Fienup.
\newblock Phase retrieval and image reconstruction for astronomy.
\newblock In H.~Stark, editor, {\em Image Recovery: Theory and Application},
  pages 231--275. Academic Press, Orlando, FL, 1987.

\bibitem{1978object}
J.~R. Fienup.
\newblock Reconstruction of an object from the modulus of its fourier
  transform.
\newblock {\em Optics letters}, 3(1):27--29, 1978.

\bibitem{2014ast}
R.~A. Gonsalves.
\newblock Perspectives on phase retrieval and phase diversity in astronomy.
\newblock In E.~Marchetti, L.~Close, and J.~Veran, editors, {\em ADAPTIVE
  OPTICS SYSTEMS IV}, volume 9148 of {\em Proceedings of SPIE}, 2014.

\bibitem{2022stft}
P.~Grohs and L.~Liehr.
\newblock On foundational discretization barriers in {STFT} phase retrieval.
\newblock {\em J. Fourier Anal. Appl.}, 28(2):Paper No. 39, 21, 2022.

\bibitem{2025stft}
P.~Grohs and L.~Liehr.
\newblock Phaseless sampling on square-root lattices.
\newblock {\em Found. Comput. Math.}, 25(2):351--374, 2025.

\bibitem{2025stft1}
P.~Grohs, L.~Liehr, and M.~Rathmair.
\newblock Multi-window {STFT} phase retrieval: lattice uniqueness.
\newblock {\em J. Funct. Anal.}, 288(3):Paper No. 110733, 23, 2025.

\bibitem{2019gabor}
P.~Grohs and M.~Rathmair.
\newblock Stable {G}abor phase retrieval and spectral clustering.
\newblock {\em Comm. Pure Appl. Math.}, 72(5):981--1043, 2019.

\bibitem{2013Quantum}
T.~Heinosaari, L.~Mazzarella, and M.~M. Wolf.
\newblock Quantum tomography under prior information.
\newblock {\em Comm. Math. Phys.}, 318(2):355--374, 2013.

\bibitem{Jaganathan2016STFTPR}
K.~Jaganathan, Y.~C. Eldar, and B.~Hassibi.
\newblock Stft phase retrieval: Uniqueness guarantees and recovery algorithms.
\newblock {\em IEEE Journal of Selected Topics in Signal Processing},
  10(4):770--781, June 2016.

\bibitem{2015quantum}
M.~Kech and M.~Wolf.
\newblock From quantum tomography to phase retrieval and back.
\newblock In {\em 2015 International Conference on Sampling Theory and
  Applications (SampTA)}, pages 173--177, 2015.

\bibitem{1988optics}
T.~I. Kuznetsova.
\newblock On the phase retrieval problem in optics.
\newblock {\em Soviet Physics Uspekhi}, 31(4):364, apr 1988.

\bibitem{2021Conjugate}
C.-K. Lai, F.~Littmann, and E.~S. Weber.
\newblock Conjugate phase retrieval in {P}aley-{W}iener space.
\newblock {\em J. Fourier Anal. Appl.}, 27(6):Paper No. 89, 23, 2021.

\bibitem{Li2017PhaseRetrieval}
L.~Li, C.~Cheng, D.~Han, Q.~Sun, and G.~Shi.
\newblock Phase retrieval from multiple-window short-time fourier measurements.
\newblock {\em IEEE Signal Processing Letters}, 24(4):372--376, Apr. 2017.

\bibitem{2021stft1}
R.~Li, B.~Liu, and Q.~Zhang.
\newblock Uniqueness of {STFT} phase retrieval in shift-invariant spaces.
\newblock {\em Appl. Math. Lett.}, 118:Paper No. 107131, 6, 2021.

\bibitem{2024stlct}
R.~Li and Q.~Zhang.
\newblock Uniqueness of phase retrieval with short-time linear canonical
  transform.
\newblock {\em Anal. Appl. (Singap.)}, 22(7):1181--1193, 2024.

\bibitem{2004entire}
J.~N. McDonald.
\newblock Phase retrieval and magnitude retrieval of entire functions.
\newblock {\em J. Fourier Anal. Appl.}, 10(3):259--267, 2004.

\bibitem{2014cry}
J.~Navaza.
\newblock The phase problem in crystallography.
\newblock {\em Gac. R. Soc. Mat. Esp.}, 17(4):705--715, 2014.

\bibitem{pesenson2025notes}
I.~Z. Pesenson.
\newblock Notes on bernstein spaces, sampling, boas interpolation formulas and
  their extensions to banach spaces, 2025.
\newblock 29 pages.

\bibitem{2018cry}
S.~Pinilla, H.~Garc\'{\i}a, L.~D\'{\i}az, J.~Poveda, and H.~Arguello.
\newblock Coded aperture design for solving the phase retrieval problem in
  {X}-ray crystallography.
\newblock {\em J. Comput. Appl. Math.}, 338:111--128, 2018.

\bibitem{1983cry}
O.~E. Piro.
\newblock Information theory and the ``phase problem'' in crystallography.
\newblock {\em Acta Cryst. Sect. A}, 39(1):61--68, 1983.

\bibitem{2015Optical}
Y.~Shechtman, Y.~C. Eldar, O.~Cohen, H.~N. Chapman, J.~Miao, and M.~Segev.
\newblock Phase retrieval with application to optical imaging: A contemporary
  overview.
\newblock {\em IEEE Signal Processing Magazine}, 32(3):87--109, 2015.

\bibitem{1963optics}
A.~Walther.
\newblock The question of phase retrieval in optics.
\newblock {\em Optica Acta}, 10:41--49, 1963.

\end{thebibliography}

\end{document}